\documentclass[12pt]{article}
\usepackage{graphicx} 
\usepackage{fullpage}
\usepackage{amsfonts,amsmath,amssymb,amsthm}
\usepackage[colorlinks]{hyperref}
\usepackage{comment}

\newtheorem{theorem}{Theorem}[section]

\newtheorem{remark}[theorem]{Remark}
\newtheorem{prop}[theorem]{Proposition}

\newtheorem{cor}[theorem]{Corollary}

\DeclareMathOperator{\Tr}{Tr}

\newcommand{\ZZ}{\mathbb{Z}}
\newcommand{\NN}{\mathbb{N}}
\newcommand{\RR}{\mathbb{R}}
\newcommand{\TT}{\mathbb{T}^2}

\newcommand{\B}{\mathcal{B}}

\newcommand{\Hc}{\mathcal{H}}
\newcommand{\LG}{\mathcal{L}}

\newcommand{\Td}{\dot{T}}
\newcommand{\LGd}{\dot{\LG}}
\newcommand{\rr}{r}
\newcommand{\p}{2}
\newcommand{\q}{1}

\usepackage[dvipsnames]{xcolor}

\title{Optimal linear response for Anosov diffeomorphisms}
\author{Gary Froyland\thanks{g.froyland@unsw.edu.au, corresponding author} and Maxence Phalempin\thanks{m.phalempin@unsw.edu.au} \\ School of Mathematics and Statistics \\ UNSW Sydney, Sydney NSW 2052, Australia}

\date{}

\begin{document}

\maketitle
\begin{abstract}
It is well known that an Anosov diffeomorphism $T$ enjoys linear response of its SRB measure with respect to infinitesimal perturbations $\dot{T}$.
For a fixed observation function $c$, we develop a theory to optimise the response of the SRB-expectation of $c$.
Our approach is based on the response of the transfer operator on the anisotropic Banach spaces of Gou\"ezel--Liverani.
We prove that the optimising perturbation $\dot{T}$ is unique for non-degenerate response functions and provide explicit expressions for the Fourier coefficients of $\dot{T}$.
We develop an efficient Fourier-based numerical scheme to approximate the optimal vector field $\dot{T}$, along with a proof of convergence.
The utility of our approach is illustrated in two numerical examples, by localising SRB measures with small, optimally selected, perturbations.
\end{abstract}

\section{Introduction}\label{secintro}

The expanding and contracting behaviour of hyperbolic dynamical systems $T_0:M\to M$ have made such systems classical models of chaotic behaviour, and their statistical properties have been widely studied (see \cite{Hopf60,Smale73,Anosov_Sinai_1967} or \cite{katok} for an overview). 
Transitive uniformly hyperbolic systems exhibit a unique physical measure $f_0$, the Sinai--Ruelle--Bowen (SRB) measure, which describes the long-term distribution of trajectories initialised in a Lebesgue-positive-measure set of points.
The SRB measure was introduced in the fundamental works \cite{Sinai72,Bowen75,Ruelle76} and was extended to systems with more complicated or weaker hyperbolicity assumptions like the Sinai Billiard or H\'enon map in \cite{young98}; see \cite{young02} for a survey.
If such systems are perturbed in a suitable way, they exhibit \textit{linear response} in the sense that a smooth perturbation of the transformation $T_0$ into a transformation $T_\delta$ induces a smooth perturbation of the SRB measure $f_0$ into $f_\delta$.
The notion of linear response for dynamical systems was first introduced by Ruelle \cite{Ru97} and was later generalized to systems exhibiting weaker forms of hyperbolicity (see \cite{Ba14} for an overview).
In this work we are concerned with  linear response of Anosov maps of the torus, a classical class of hyperbolic systems inspired by the work of Anosov \cite{Anosov67} and whose ergodic properties were investigated by  \cite{Sinai72},\cite{Ruelle76} and \cite{Bowen75}.
We focus on \textit{optimal linear response} \cite{DrFr18,AnFrGa22,FrGa25,GaNi25}, which in the present setting identifies the perturbation of $T_0$ that optimises the linear response of the expectation of a specified observation $c:M\to \mathbb{R}$ with respect to the SRB measure.
By choosing the function $c$ appropriately, one may engineer perturbations of $T_0$ that achieve a desired long-term system effect. 

To detail these notions we introduce the following formal setting. 
We consider a family of transitive  {$C^{\rr+1}$} Anosov maps of the two-torus, $\{T_\delta:\mathbb{T}^2\to\mathbb{T}^2\}_{\delta\in[0,\bar\delta]}$ for some $\bar\delta>0$ { {sufficiently} small} and {$\rr=4$}.
We denote the SRB measure corresponding to $T_\delta$ by $f_\delta$.
We assume that this family is differentiable with respect to $\delta$ at $\delta=0$ in the following sense:  there is a $\dot{T} \in {C^{\rr+1}}(\mathbb{T}^2,\mathbb{R})$  such that,
\begin{equation}\label{Ttaylorintro}
T_{\delta }=T_{0}+\delta\cdot \dot{T}+o_{C^{5}}(\delta).
\end{equation}
We take a functional analytic approach and use the Banach spaces $\mathcal{B}^{p,q}$ devised in \cite{GL06}.
The family of maps in \eqref{Ttaylorintro} gives rise to a family of transfer operators $\{\mathcal{L}_\delta:\mathcal{B}^{2,1}\to\mathcal{B}^{2,1}\}_{\delta\in[0,\bar\delta]}$,
which leads naturally to a notion of linear response for the transfer operators.
Theorem \ref{prop:taylorcon} proves the existence of a continuous operator $\dot{\mathcal{L}}:\mathcal{B}^{2,1}\to\mathcal{B}^{1,2}$ satisfying 
\begin{equation*}
\mathcal{L}_{\delta }=\mathcal{L}_{0}+\delta\cdot \dot{\mathcal{L}}+o(\delta),
\end{equation*}
and derives an explicit expression for $\dot{\mathcal{L}}$ in terms of $T_0, \dot{T}$, and $\mathcal{L}_0$, namely 
\begin{equation}
    \label{dotLexpr}
      \LGd f(\cdot)=-\LG_0\left(\nabla \cdot \left(f(\cdot)(D_\cdot T_0)^{-1}\circ\dot T(\cdot)\right)\right),
\end{equation}
where $\nabla \cdot$ stands for the divergence operator.
In Theorem \ref{thm:deriv} we further show that $\dot{\mathcal{L}}(\dot{T})$ is a continuous function of $\dot{T}$ from $C^{\rr+1}(\mathbb{T}^2,\mathbb{R}^2)$ to $L(\B^{2,1},\B^{1,2})$, a fact that will be crucial for our later optimisation of response.

A fundamental question in linear response is the response of the SRB measure $f_\delta\in \mathcal{B}^{2,1}$ of $T_\delta$ with respect to the parameter $\delta$ \cite{Ru97}.
In Theorem \ref{thmresp} we use the expression \eqref{dotLexpr} and classical perturbation results from \cite{KL99} to provide a simple demonstration of the existence of linear response of the invariant measure $\dot{f}:=\frac{d}{d\delta}f_\delta|_{\delta=0}$.
Considering this response as a function of $\dot{T}$, so that $R(\dot{T})=\dot{f}$, in Theorem \ref{thm:deriv} we show that the response function 
\begin{align}
\label{eq:responseintro}
R(\dot T)=(I-\mathcal{L}_0)^{-1}\dot{\mathcal{L}}(\dot{T})(f_0)   
\end{align}
is a continuous function of $\dot{T}$ from $C^{\rr+1}(\mathbb{T}^2,\mathbb{R}^2)$ to $\mathcal{B}^{1,2}$. 
We then turn to the existence and optimisation of linear response of the integrals $\frac{d}{d\delta}\int c\cdot f_\delta = \int c\cdot \dot{f}$, where $c$ is a regular observation function.
Given an observable $c\in C^3$, 
we consider the problem of maximising the response $J(\dot{T}):=\int c\cdot R(\dot{T})$
over all unit-size perturbations $\dot{T}$ in a suitable Hilbert space $\mathcal{H}$, such as $\mathcal{H}=H^7(\mathbb{T}^2,\mathbb{R}^2)$, contained in $C^5$ perturbations $\dot{T}$.
That is, we wish to solve 
\begin{eqnarray}
\nonumber\max_{\dot{T}\in\mathcal{H}}&& \int c\cdot R(\dot{T})\\
\nonumber\mbox{subject to}&&\|\dot{T}\|_{\mathcal{H}}=1.    
\end{eqnarray}

The continuity of the response functional $J$, and the boundedness and strict convexity of the feasible set $\{\dot{T}\in\mathcal{H}:\|\dot{T}\|_\mathcal{H}\le 1\}$ over which we optimise, immediately guarantees the existence of a unique optimal perturbation $\dot{T}$ (Proposition \ref{propoptigen}).
Using Riesz representation, in Theorem \ref{mainthm1} we derive explicit expressions for the Fourier coefficients of this optimal vector field $\dot{T}$.
From these coefficient expressions, in Section \ref{sec:numerics} we detail a highly efficient numerical scheme, based on the Fejer kernel approach of \cite{CF20}, to compute estimates of the optimal $\dot{T}$.
The convergence of the estimated optimal perturbations to the true optimal perturbation $\dot{T}$ is provided in Theorem \ref{thmmolli} and Corollary \ref{corfejermol}.

We illustrate the effectiveness of our approach in two case studies in Section \ref{sec:casestudy}.
First, we consider the Arnold cat map and a function $c$ with large values concentrated about the fixed point at the origin.
Thus, we expect the optimal perturbation $\dot{T}$ to be such that maps $T_0+\delta\cdot\dot{T}$ for small $\delta>0$ have SRB measures with density preferentially concentrated about the fixed point. 
We describe the dynamically interesting strategy taken by the computed optimal $\dot{T}$ -- shown in Figure \ref{linearvf} -- and verify that the SRB measure of $T_0+\delta\cdot\dot{T}$ is indeed concentrated about the fixed point for small $\delta>0$ (in Figure \ref{linearsrbpert}).
Our second case study concerns a nonlinear Anosov map $T_0$, for which we estimate the SRB measure, and then select an observation $c$ with large values in the vicinity of a period-two orbit of $T_0$.
We expect that the optimal perturbation $\dot{T}$ will have the effect of concentrating the SRB measure of $T_0+\delta\cdot\dot{T}$ nearby the period-two orbit and this is verified in Figure \ref{nonlinearsrbpert}.
Both case studies are examples of how one may use this optimisation approach to control the long-term behaviour of the dynamics through small perturbations that are carefully chosen to achieve the desired effect. 

If one substitutes the expression \eqref{dotLexpr} for $\dot{\mathcal{L}}$ into the response function \eqref{eq:responseintro}, one obtains an expression equivalent to the usual ``metaformula'' for invariant measure response discussed in \cite{Ba14}. 
The main difference is that the ``metaformula'' does not contain an explicit expression for $\dot{\mathcal{L}}$.
Our expression \eqref{dotLexpr} for $\dot{\mathcal{L}}$ is the natural generalisation to Anosov maps of the expression given in \cite{GaPo17} for expanding circle maps, where minimum-Sobolev-norm perturbations $\dot{T}$ were found for a given (fixed) response.
Equation \eqref{dotLexpr} is  superficially different (but equivalent) to the expression for $\dot{\mathcal{L}}$ provided in \cite{Po19} in the context of Anosov diffeomorphisms.
We remark that to our knowledge, our $C^1$-regularity hypothesis on the map $\delta \mapsto T_\delta$ is weaker than existing work in the Anosov setting.

Optimal linear response has been investigated in a series of papers, progressing from finite state Markov chains in \cite{DrFr18}, to dynamical systems representable by a  Hilbert--Schmidt integral operator \cite{AnFrGa22} (including optimising both stochastic and deterministic components of annealed random dynamical systems), and deterministic expanding maps of the circle \cite{FrGa25}. All of these papers optimise the linear response of the expected value of observations $c$ as well as the linear response of isolated eigenvalues associated to Markov chains or transfer operators.
Other recent work includes optimising the linear response of the expectation of an observation for annealed transfer operators arising from stochastic differential equations, following ideas from \cite{AnFrGa22}.
Perhaps the closest work to the present one is the use of the so-called fast adjoint response method to estimate the optimal linear response of expectations of observations for uniformly hyperbolic dynamics \cite{GaNi25}. 
In contrast to \cite{GaNi25}, we work on the full domain of the hyperbolic maps $T_\delta$ and when considering the corresponding transfer operators $\mathcal{L}_\delta$, we need to use the anisotropic Banach spaces mentioned earlier.
While there is some technical cost to this, there are several important benefits to having an explicit expression \eqref{dotLexpr} for $\dot{\mathcal{L}}$ in terms of $\mathcal{L}_0$, $D_xT_0$ and $\dot{T}$.
This (i) allows one to use the standard linear response expression \eqref{eq:responseintro}, 
(ii) enables a convergence theory for the numerical scheme, and (iii) opens the possibility to derive other types of optimal linear responses, beyond the response of expectation of an observation $c$.

Because we consider $\mathcal{L}_0$ to act on functions on the full domain of $T_0$, we obtain $f_0$ as the leading eigenfunction of $\mathcal{L}_0$, a fact that is also computationally beneficial.
The explicit expression for the linear response then leads naturally via Riesz representation arguments to explicit expressions for Fourier coefficients of an optimal infinitesimal perturbation $\dot{T}$ in a Fourier basis on the full space.
The convergence theory of \cite{CF20} provides a basis for convergence of the optimal $\dot{T}$ to the true object as the number of Fourier modes in the construction of $\mathcal{L}_0$ increases. 
Our computational implementation is extremely fast with each Fourier coefficient of the optimal $\dot{T}$ taking milliseconds to compute. In contrast to \cite{GaNi25}, our algorithm is deterministic and does not rely on sampling, which may assist construction of computer-assisted error bounds.

All of our theoretical results and our numerical scheme straightforwardly extends to higher-dimensional tori, although this would commensurately increase the cost of the computations. 
Because our theory mostly uses the anisotropic Banach spaces that are well adapted to Anosov flows on compact oriented Riemannian manifolds, the theoretical results introduced in Section \ref{sec:derivLR} and Theorem \ref{thm:deriv} can be adapted to these settings.

The study of linear response -- and optimal linear response in particular -- is important in many applied domains, such as climate science (e.g.\ \cite{HaMa10,GhLu20}) and social science, where optimally selected perturbations can achieve desired outcomes in long-term behaviour for minimal cost. 
Indeed, because our numerical approach involves kernel convolution, one may view our estimates as those corresponding to a map $T_0$ perturbed by i.i.d.\ noise with density given by the localised kernel.
Thus, even when our methodology is applied to maps $T_0$ that are not Anosov and for which one does not have a formal functional analytic theory, as long as the dynamics supports linear response, one may expect meaningful results from optimising deterministic perturbations.

\section{Background}
\label{sec:background}

\subsection{The map perturbations}
For some small $\bar\delta>0$ and $r=4$
we consider a family of {$C^{\rr+1}$} Anosov maps of the two-torus, $\{T_\delta:\mathbb{T}^2\to\mathbb{T}^2\}_{\delta\in[0,\bar\delta]}$.
This family is differentiable with respect to $\delta$ at $\delta=0$ in the following sense: 
\begin{equation}\label{Ttaylor}
T_{\delta }(x)=T_{0}(x)+\delta \dot{T} (x)+o_{C^{5}}(\delta)
\end{equation}
where $\dot{T} \in {C^{\rr+1}}(\mathbb{T}^2,\mathbb{R})$ and we say a family of functions $F_\delta\in {C^{\rr+1}}(\mathbb{T}^2,\mathbb{R})$ is $o_{C^{5}}(\delta)$
if
\begin{align*}
\lim_{\delta\to 0}\frac{||F_\delta||_{C^{5}}}{\delta}=0.
\end{align*}

\subsection{The  Banach spaces and the transfer operator}\label{secbanach}
The transfer operator for $T_\delta$ acting on continuous functions is as usual defined by
\begin{equation}\label{eq:defldelta}
\mathcal{L}_\delta h := h\circ T_\delta^{-1}/|\det(DT_\delta)|\circ T_\delta^{-1}=h\circ T_\delta^{-1}|\det(D(T_\delta^{-1}))|.
\end{equation}
The action of $\mathcal{L}_\delta$ on distributions of order $r$ is defined by duality: 
$$
\int_{\mathbb{T}^2} \varphi \cdot \mathcal{L}_\delta h = \int_{\mathbb{T}^2} \varphi\circ T_\delta \cdot  h, \qquad\mbox{for all $\varphi\in C^\infty(\mathbb{T}^2,\mathbb{C})$.}
$$
For any $p\in \NN$ and $q\in \RR_+$ such that $p+q<r$
we can associate to $T_0$ the anisotropic Banach space $(\mathcal{B}^{p,q},\|\cdot\|_{\B^{p,q}})$ introduced in \cite{GL06} as the completion of the space $C^r$ for the norm $\|\cdot\|_{\B^{p,q}}$ defined by relations $(2.1)$ and $(2.2)$ in \cite{GL06}. 
We recall in the proposition below some of the main features of these spaces;  in this paper, as mentioned previously, we will fix except when explicitly mentionned $p=2, q=1,$ and {$r=4$}.
\begin{prop}\label{lemmulti}
{Let $p\in \NN^*$ and $q\in \RR_+$ satisfy $p+q<r$.}
\begin{itemize}
\item[i)] 
The operator $\LG_0 : \B^{p,q}\mapsto \B^{p,q}$ is continuous and there is an open neighbourhood $U$ of $T_0$ in the $C^{r+1}$ topology 
such that for any $T\in U$ the transfer operator {$\mathcal{L}$ for $T$} is a continuous operator of $\B^{p,q}$ (see Lemma 7.2 in \cite{GL06} with parameters $g(\omega,x)\equiv 1$ and $\mu=\delta_\omega$).
In particular, whenever $\|T_0-T_\delta\|_{C^{\rr+1}}$ is small 
then $\mathcal{L}_\delta:\mathcal{B}^{p,q}\to \mathcal{B}^{p,q}$ is also a well-defined continuous linear operator.
\item[ii)] The inclusion $\B^{p,q}\subset \B^{p-1,q+1}$ is compact\footnote{We even have $\|f\|_{\B^{p-1,q+1}}\leq \|f\|_{\B^{p,q}}$ for $f\in \B^{p,q}$}  (see Lemma 2.1 in \cite{GL06}).
\item[iii)] From the definition of the norms  $\|\cdot\|_{\B^{p,q}}$ and $\|\cdot\|_{\B^{p-1,q+1}}$ (see Remark 4.3 in \cite{GL06}), the space $C^{r'}(\TT,\mathbb{R})$ is dense in $\B^{p,q}$ for any $r'\geq p$. 
\item[iv)] In addition (and still from equation (2.2) in \cite{GL06}) the partial derivative $\partial_i : f \in C^p \mapsto \partial_i f \in C^{p-1}$ can be uniquely extended into a continuous operator  $\partial_i : \B^{p,q}\to \B^{p-1,q+1}$. 
\item[v)] Multiplicativity of the norms (see Lemma 3.2 in \cite{GL06}) : there is a $C>0$ such that for any $f\in \B^{p,q}$ and $\phi\in C^{p+q}$, $f\phi \in \B^{p,q}$ and 
$$
\|f\phi\|_{\B^{p,q}}\leq C\|f\|_{\B^{p,q}}\|\phi\|_{C^{p+q}}\quad
\mbox{and}\quad \left|\int \phi f \right|\leq C\|f\|_{\B^{p,q}}\|\phi\|_{C^{p}}.$$
\end{itemize}
\end{prop}
Theorem 2.3 \cite{GL06} states that $\LG_0:\B^{p,q}\mapsto \B^{p,q}$ is quasi-compact with essential spectral radius bounded above by $\max\{1/\lambda^p,\nu^q\}$, where $\lambda>1$ (resp.\ $\nu<1$) is a lower (resp.\ upper) bound on the expansion (resp.\ contraction) along unstable (resp.\ stable) tangent directions.

\section{The derivative of the transfer operator and linear response}
\label{sec:derivLR}

\subsection{Expressions for the derivative of the transfer operator}
\label{sec:deriv}

For $f\in C^1(\TT,\mathbb{R}^2)$, define the divergence of $f$, $\nabla \cdot f(x):=\Tr(D_xf)$, where $\Tr(\cdot)$ is the trace operator for linear maps.
We let $\nabla \cdot :\B^{p,q} \to \B^{p-1,q+1}$ 
 denote an extension on $\B^{p,q}$ of the divergence operator (see Proposition \ref{lemmulti} item iv)).
For $E=\mathbb{R}^2$ and $A(x) \in L(E,E)$, a section of linear maps,
 $\nabla\cdot A$ stands for the classical divergence of a tensor\footnote{(see section 2C4 in \cite{lai}), certain books take an alternative definition for the divergence $\tilde \nabla$ with $\tilde \nabla \cdot A:=\nabla \cdot (A^t)$ with $A^t$ being the transpose of the matrix $A$.}  i.e.\ for any $x\in \TT$ and $a\in E$, 

\begin{equation*}
    \label{def:div}
 (\nabla \cdot A(x))\boldsymbol{\cdot} a:=\nabla \cdot \left(A(x)\boldsymbol{\cdot} a\right),
 \end{equation*}
where the larger dot $\boldsymbol{\cdot}$ denotes the relevant linear action on $a$.
The following theorem states the regularity of the transfer operator $\LG_\delta$ and its linear response along the lines of Theorem 2.7 from \cite{GL06}, but only requires the map $\delta \mapsto T_\delta$ to be $C^1$ in comparison to the $C^2$ regularity required by \cite{GL06}.

\begin{theorem}\label{prop:taylorcon} 
Let $p\in \NN\backslash \{0\}$ and {$q\geq 0$} such that $p+q=3$.
Then there is $\tilde \delta>0$ such that  
the map $\delta \in [0,\tilde \delta) \mapsto \LG_\delta\in \mathcal{B}^{p,q}$ is continuous, and there is a $C>0$ satisfying
{\begin{align}\label{eqmainthmcondelta}
\|\LG_\delta-\LG_0\|_{\B^{p,q}\to \B^{p-1,q+1}}\leq C\delta.    
\end{align}
}
When $p=2$ and $q=1$, the operator $\LG_\delta : \B^{2,1}\mapsto \B^{0,3}$  satisfies the following Taylor relation {in the parameter $\delta$:}
\begin{align}
    \label{eq:Ldotest}
        \sup_{f\in \B^{\p,\q},\|f\|_{\p,\q}\leq 1}\left\|\dot{\mathcal{L}}f-\frac{\LG_\delta f-\LG_0f}{\delta}\right\|_{0,3}={o(1)},
\end{align}
and $\dot{\mathcal{L}}:\mathcal{B}^{2,1}\to\mathcal{B}^{1,2}$ is continuous.
Finally, the derivative of the transfer operator \\
$\dot{\mathcal{L}} : \B^{2,1} \to \B^{1,2}\subset \B^{0,3}$  has the explicit representations:
\begin{align}
    \LGd f (x)&:=-\LG_0\left(\nabla \cdot \left(f(x)(D_x T_0)^{-1}(\dot T(x))\right)\right),\label{eq:divrep}\\
    &= -\LG_0\left(\langle D_xf, (D_xT_0)^{-1}\left(\dot T(x)\right)\rangle+ f(x)\langle \nabla\cdot (D_xT_0)^{-1},\dot T(x)\rangle+ f(x)\Tr((D_xT_0)^{-1}\circ D_x\dot T)\right) \label{eq:divexp}
\end{align}
\end{theorem}

\begin{proof}
\,
\paragraph{Step 1: Derivation of \eqref{eq:divrep} and \eqref{eq:divexp} for smooth $f$.}
We assume for now that $f \in C^{\infty}(\TT,\mathbb{R})$ and we will later expand the following computations of $\LGd$ to any $f \in \B^{\p,\q}$ by density of $C^{\infty}(\TT,\mathbb{R})$ in $\|\cdot \|_{\p,\q}$.
\begin{align}
\dot{\LG}f(x)&:=\partial_{|\delta=0}\frac{f\circ T_\delta^{-1}(x)}{|\det(D_{T_\delta^{-1}x}(T_\delta))|}\label{eq:ldotderiv1}\\
\nonumber&=|\det(D_x(T_0^{-1}))|\partial_{|\delta=0} \left(f\circ T_\delta^{-1}(x)\right)+(f\circ T_0^{-1}(x))\partial_{|\delta=0}\left(|\det(D_x(T_\delta^{-1}))|\right)\\
\label{eq:dotTprelim}&=|\det(D_x(T_0^{-1}))|D_{T_0^{-1}x}f\circ\left(\partial_{|\delta=0} T_\delta^{-1}(x)\right)+(f\circ T_0^{-1}(x))\partial_{|\delta=0}\left(|\det(D_x(T_\delta^{-1}))|\right).
\end{align}
We will thus compute the respective derivatives $\partial_{|\delta=0} T_\delta^{-1}$ and $\partial_{|\delta=0}\left(|\det(D(T_\delta^{-1}))|\right)$ in $C^3$ 
and show that
\begin{align}
T_\delta^{-1}(x)=&T_0^{-1}(x)+\delta\partial_\delta (T_\delta^{-1})(x)+{o_{C^3}(\delta)}\label{eqdifftmoins1}\\
|\det(D_x(T_\delta^{-1}))|&=|\det(D_x(T_0^{-1}))|+\delta\partial_{|\delta=0}\left(|\det(D_x(T_\delta^{-1}))|\right)+{ o_{C^3}(\delta)}.\label{eqdifftmoins2}
\end{align}

We start with equation \eqref{eqdifftmoins1}, the local inversion technique
applied to\footnote{The map $T$ is taken in $C^{r+1}$ in order to keep $o(\delta)$ control in the $C^3$ topology in the following developments.
} the $C^{1}$ map $\Psi : C^4(\TT)\times C^{3}(\TT) \mapsto C^{3}(\TT)$ given by $\Psi (T_1,T_2)=T_1\circ T_2$, which satisfies $\Psi( T_\delta,T_\delta^{-1})=I$.
This ensures the existence of a $C^1$ map $\tilde \Psi: T\in C^4\mapsto T^{-1}\in C^3$ with derivative at $T\in C^4$ in the direction $H\in C^4$ given by
\begin{align}
   \partial_T\tilde \Psi(H)&=-\left(\partial_2\Psi(T,T^{-1})\right)^{-1}\circ\partial_{1}\Psi(T,T^{-1})\nonumber\\
   &=-(D_{T^{-1}x}T)^{-1}(H\circ T^{-1}),\label{eqdeltatmoinsun}
\end{align}
where the following derivatives have been computed explicitly:
\begin{align}
\label{D1}
&\left(\partial_1\Psi(T,T^{-1})\right)(H)=H\circ T^{-1}\\
&(\partial_{2}\Psi(T,T^{-1}))^{-1}(H)=\left(D_{T^{-1}x}T\right)^{-1}(H)=D_x(T^{-1})(H).\label{eqddeuxpsi}
\end{align}
Since $T\in C^4\mapsto T^{-1}\in C^3$ is a $C^1$ map in the $C^3$ topology, 
\begin{align}
T_\delta^{-1}=T_0^{-1}+\delta\partial_\delta (T_\delta^{-1})+o_{C^3}(\delta)   
\end{align}
 where $\partial_\delta (T_\delta^{-1})$ can be detailed as follows :
Taking the $\delta$-derivative of $\Psi(T_\delta,T_\delta^{-1})$ we obtain 
$$
\partial_1\Psi(T_0,T_0^{-1})\circ \partial_{|\delta=0} T_\delta + \partial_2\Psi(T_0,T_0^{-1})\circ \partial_{|\delta=0} T_\delta^{-1} = 0,
$$
thus 
using relations {\eqref{D1} and}  \eqref{eqddeuxpsi} with $\partial_{|\delta=0}T_\delta=\Td$,
\begin{align}
\partial_{|\delta=0} (T_\delta^{-1})(x)&=-\left(\partial_2\Psi(T_0,T_0^{-1})\right)^{-1}\circ \partial_1\Psi(T_0,T_0^{-1})\circ \partial_{|\delta=0} T_\delta \nonumber\\
&=-D_{x}(T_0^{-1})\circ \Td \circ T_0^{-1}(x).\label{eq:Tdelinv}
\end{align}
As for equation \eqref{eqdifftmoins2}, notice that $\det(D(T_\delta^{-1}))$ has constant sign in a neighbourhood of $T_0^{-1}$, thus we can ignore the absolute value $|.|$ 
and compute $\partial_{|\delta=0}\left(|\det(D(T_\delta)^{-1})|\right)$ by composition of the derivative of the $C^2$ map given by $L(\mathbb{R}^2)\ni A \mapsto \det(A)$ with $D\partial_{|\delta=0} (T_\delta^{-1})$ we computed above\footnote{indeed we can ignore the change of sign change from $|\cdot|$ in $\partial_\delta |\det(D_x(T_\delta)^{-1})|$: for an operator $A$, $D_A(\det(A))(H)=\det(A)\Tr(AH)$ and $D_A(-\det(A))(H)=-\det(A)\Tr(AH)$, thus whenever $|\det(A)|\neq 0$, $D_A(|\det(A)|)$ is well defined and $D_A(|\det(A)|)(H)=|\det(A)|\Tr(AH)$.}.
The Taylor expansion in the $C^3$ topology from \eqref{eqdifftmoins1} allows for the following expansion in the $C^2$ topology,
\begin{align*}
D_{x}T_\delta^{-1}=&D_xT_0^{-1}(x)+\delta D_x\partial_\delta (T_\delta^{-1})(x)+{o_{C^2}(\delta)}.
\end{align*}
We now need to consider $T_0$ to have its full regularity, namely $C^5$,
to make the above error term in {$o_{C^3}(\delta)$}. Assuming $T_\delta \in C^{5}$, from the formula 
\eqref{eqdeltatmoinsun} of the derivative of $\tilde \Psi :T\in  C^{5}\mapsto C^3$ we deduce that $T \in C^{5}\mapsto (x\mapsto D_x\tilde \Psi(T))\in C^3$ is $C^{1}$
and therefore
\begin{align*}
D_{x}T_\delta^{-1}=&D_xT_0^{-1}(x)+\delta D_x\partial_\delta (T_\delta^{-1})(x)+{o_{C^3}(\delta)}.
\end{align*}
Then we apply the polynomial map $\det(\cdot)$, yielding
\begin{align}
\label{eq:det}\det(D_x(T_\delta^{-1}))=\det(D_x(T_0^{-1}))+\delta \partial_{|\delta=0}\left(\det(D_x(T_\delta^{-1}))\right)+o_{C^3}(\delta).
\end{align}

Applying the definition \eqref{eq:defldelta} for the transfer operator $\LG_\delta$ and the respective Taylor expansions \eqref{eqdifftmoins1} and \eqref{eqdifftmoins2}, we arrive at the following Taylor expansion for $\LG_\delta$ at $\delta=0$:
\begin{align*}
    \LG_\delta f-\LG_0f=&\frac{f\circ T_\delta^{-1}}{|\det(D_{T_\delta^{-1}x}T_\delta)|}-\frac{f\circ T_0^{-1}}{|\det(D_{T_0^{-1}x}T_0)|}\\
    =&\frac{f\circ T_\delta^{-1}-f\circ T_0^{-1}}{|\det(D_{T_\delta^{-1}x}T_\delta)|}+f\circ T_0^{-1}\left(\frac{1}{|\det(D_{T_\delta^{-1}x}T_\delta)|}-\frac{1}{|\det(D_{T_0^{-1}x}T_0)|}\right)\\
    =&\frac{\delta\cdot(D_{T_0^{-1}x}f)\partial_{\delta}(T_\delta^{-1})+\|f\|_{C^3}o_{C^{3}}(\delta)}{|\det(D_{T_\delta^{-1}x}T_\delta)|}+f\circ T_0^{-1}\left(\delta \cdot\partial_{|\delta=0}\left(\det(D_x(T_\delta^{-1}))\right)+o_{C^3}(\delta)\right)\\
     =&\delta\cdot\LGd f+o_{C^3}(\delta)(\|f\|_{C^3}+1)\\
     &+\left(\delta\cdot(D_{T_0^{-1}x}f)\partial_{\delta}(T_\delta^{-1})+\|f\|_{C^3}o_{C^{3}}(\delta)\right)\left(\frac{1}{|\det(D_{T_0^{-1}x}T_0)|}-\frac{1}{|\det(D_{T_\delta^{-1}x}T_\delta)|}\right)\\
    =&\delta\cdot\LGd f+o_{C^3}(\delta)(\|f\|_{C^3}+1),
\end{align*}
where in the final line we recall that $\dot{\LG}$ is given by \eqref{eq:dotTprelim}.
Thus for $f\in C^\infty(\TT)$, $\LG_\delta f=\LG_0f+\delta \LGd f+o_{C^3}(\delta)(\|f\|_{C^3}+1)$.
Because this Taylor expansion is uniform in  $C^2$ we can interchange the integral and derivative in the following computation and obtain the explicit formula of $\partial_{\delta=0}\LG_{\delta}f=\LGd f$ for $f\in C^\infty(\TT)$ :
Let $f,g\in C^\infty(\TT)$, then

\begin{align}
\int_{\mathbb{T}^2} \partial_{\delta=0}
\big( \LG_\delta f(x) \big)\, g(x)
&= \partial_{\delta=0}
\left( \int_{\mathbb{T}^2} \LG_\delta f(x)\, g(x)\right)\\
&= \partial_{\delta=0} \left( \int_{\mathbb{T}^2} f(x)\, g\circ T_\delta (x) \right)\\
&= \int_{\mathbb{T}^2} f(x)\, \partial_{\delta=0}(g\circ T_\delta(x))\\
&= \int_{\mathbb{T}^2} f(x)\, \partial_{\delta=0}(g\circ T_0\circ T_0^{-1}\circ T_\delta(x))\label{eq:dualderiv1}\\
&= \int_{\mathbb{T}^2} f(x)\, D_{x}(g\circ T_0)\left(D_{T_0x}T_0^{-1}( \partial_{\delta=0}T_\delta(x))\right)\label{eq:dualderiv2}\\
&= \int_{\mathbb{T}^2} f(x)\,  D_{x}(g\circ T_0)\circ ((D_x T_0)^{-1} (\dot{T}(x))) \\
&=-\int_{\mathbb{T}^2} \LG_0\left(\nabla\cdot \left(f(x)  (D_x T_0)^{-1} (\dot{T}(x))\right)\right) g
\end{align}
where we obtained the last line by integrating by part using the Green Formula {(see e.g.\ the formula (9.3) from \cite{Taylor11})} for the derivative of $g$ along the vector field $Y(x) = (D_x T_0)^{-1}( \dot{T}(x))$. The passage from line \eqref{eq:dualderiv1} to \eqref{eq:dualderiv2} used the chain rule and the fact that $D_{T_0x}T_0^{-1}=(D_{x}T_0)^{-1}$.
Finally, we expanded formula \eqref{eqdifftmoins1} using the formula for the divergence of tensor (see formula (2C4.3) in \cite{lai}): for $\mathfrak{T}$ a tensor field and $X$ a vector field,
\begin{equation}
   \nonumber
\nabla \cdot (\mathfrak{T}X)=\langle (\nabla \cdot \mathfrak{T}),X\rangle +\Tr( \mathfrak{T}\circ \nabla X),
\end{equation}
which with $\mathfrak{T}(x)=(D_xT_0)^{-1}$ and $X(x)=\dot{T}(x)$ yields 
\begin{align}
\nonumber\LGd f (x)&=-\LG_0\left(\nabla \cdot \left(f(x)(D_x T_0)^{-1}\circ\dot T(x)\right)\right) \\
\label{eq:expandeddivrep}  &= -\LG_0\left(\langle D_xf, (D_xT_0)^{-1}\circ\dot T(x)\rangle+ f(x)\langle \nabla\cdot (D_xT_0)^{-1},\dot T(x)\rangle+ f(x)\Tr((D_xT_0)^{-1}\circ D_x\dot T)\right),
\end{align}
thus proving equation \eqref{eqdifftmoins2} for any $f \in C^{\infty}(\TT,\mathbb{R})$.

\paragraph{Step 2: Well-definedness and continuity of $\dot{\mathcal{L}}:\mathcal{B}^{2,1}\to\mathcal{B}^{1,2}$.}
Notice that according to item iv) from Proposition \ref{lemmulti}, the map $D_xf$ extends to a continuous map $f\in \B^{2,1} \mapsto D_xf \in \B^{1,2}$. 
Using the continuity of $\LG_0: \B^{p,q}\mapsto \B^{p,q}$ for $p+q=3$ (see Proposition \ref{lemmulti}), we deduce that the operator given by equation \eqref{eq:divexp} extends to a continuous operator  $\LGd : \B^{2,1}\to \B^{1,2}$ as well as a continuous operator $\LGd : \B^{1,2}\to \B^{0,3}$ 
\paragraph{Step 3: Derivation of \eqref{eqmainthmcondelta}.}
To prove equation \eqref{eqmainthmcondelta}, we introduce for $f\in C^\infty$ the map $t\mapsto \tilde \LG_t f:=f\circ (tT_{\delta}^{-1}+(1-t)T_0^{-1})$.
The {zeroth-order} Taylor expansion of this map about $t=0$ and evaluated at $t=1$ is
$f\circ T_\delta^{-1} =\tilde \LG_1 f
    = f\circ T_0^{-1} +G_0,
$
where 
\begin{align*}
G_0(\cdot):&=\int_0^1 (1-s)D_{sT_\delta^{-1}(\cdot)+(1-s)T_0^{-1}(\cdot)}f(\cdot)((T_\delta^{-1}-T_0^{-1})(\cdot))\ ds \nonumber\\
&=\int_0^1 (1-s) \tilde \LG_s\left(\partial_1f\right)(\cdot) (T_\delta^{-1}(\cdot)-T_0^{-1}(\cdot))_1\ ds\\
\nonumber&+\int_0^1 (1-s) \tilde \LG_s\left(\partial_2f\right)(\cdot) (T_\delta^{-1}(\cdot)-T_0^{-1}(\cdot))_2\ ds.
\end{align*}

According to\footnote{more precisely it is a consequence of Lasota--Yorke inequality (7.7) from \cite{GL06} when choosing $g(x,\omega)=\frac{1}{|\det(D_x(sT_\delta^{-1}(\cdot)+(1-s)T_0^{-1})|}$ and $\mu:=\delta_{sT_\delta^{-1}(\cdot)+(1-s)T_0^{-1}}$ in their notation.}
Lemma 7.2 in \cite{GL06}, the family $\tilde \LG_s$ are continuous on $\B^{2,1}$ 
and
 by equation \eqref{eqdifftmoins1}, $\|T_\delta^{-1}-T_0^{-1}\|_{C^3}\leq C\delta$. Thus using items v), i), and iv) from Proposition \ref{lemmulti}, 
\begin{align*}
\|G_0\|_{p-1,q+1}&\leq \int_0^1 (1-s) \|\tilde \LG_s\left(\partial_1f\right)\|_{p-1,q+1} \|(T_\delta^{-1}(\cdot)-T_0^{-1}(\cdot))_1\|_{C^3}ds\\
&+\int_0^1 (1-s) \|\tilde \LG_s\left(\partial_2f\right)\|_{p-1,q+1} \|(T_\delta^{-1}(\cdot)-T_0^{-1}(\cdot))_2\|_{C^3}ds\\
&\leq C \delta \int_0^1 (1-s) \|\tilde \LG_s\left(\partial_1f\right)\|_{p-1,q+1} ds+C \delta \int_0^1 (1-s) \|\tilde \LG_s\left(\partial_2f\right)\|_{p-1,q+1} ds\\
&\leq 2C_0 \delta \| f\|_{p,q}.
\end{align*}

Recall from step 1 that we prove that $T\in C^5\mapsto D_xT^{-1}\in C^3$ is a $C^1$ map, thus the map $\delta\mapsto \det(D(T_\delta^{-1}))$ is 
 differentiable at $\delta=0$.
Similarly, {recalling that $D_\cdot T_0^{-1}$ has constant sign for $T_\delta$ in a neighborhood of $T_0$,}
the map $\delta\mapsto |\det(D(T_\delta^{-1}))|$ is {differentiable at $\delta=0$}, and noting that $\mathcal{L}_\delta=|\det(D(T_\delta^{-1}))|\tilde{\mathcal{L}}_1$, we deduce
\begin{align*}
    \|\LG_\delta f-\LG_0 f\|_{p-1,q+1}&\leq \||\det(D(T_\delta^{-1}))|\tilde \LG_1 f-\LG_0 f\|_{p-1,q+1}\\
    &\leq\||\det(D(T_\delta^{-1}))|-|\det(D(T_0^{-1}))|\|_{C^3}\|\tilde \LG_1 f\|_{p-1,q+1}\\
    &\qquad+\||\det(D(T_0^{-1}))|\tilde \LG_1 f-\LG_0 f\|_{p-1,q+1}\\
    &\leq {O(\delta)}\|\partial_\delta |\det(D(T_\delta^{-1}))|\|_{C^3}\|\tilde \LG_1 f\|_{p-1,q+1}+\||\det(D(T_0^{-1}))|\|_{C^3}\|G_0\|_{p-1,q+1}\\
    &\leq {O(\delta)}\|f\|_{p,q}.
\end{align*}

Since $C^\infty$ is dense into $\B^{p,q}$ the inequality above extends to $\mathcal{B}^{p,q}$. This proves \eqref{eqmainthmcondelta}, namely there is $C>0$ such that
\begin{align*}
\|\LG_\delta-\LG_0\|_{\B^{p,q}\to \B^{p+1,q-1}}\leq C\delta. 
\end{align*}

\paragraph{Step 4: Derivation of \eqref{eq:Ldotest}.}
Finally, we focus on the proof of \eqref{eq:Ldotest}, {where $f\in \mathcal{B}^{2,1}$.} 
{We first consider $f\in C^\infty$ and} control $\LG_\delta f:= \frac{f\circ T_\delta^{-1}}{|\det(DT_\delta)|}$ in several steps.
First fix $\delta>0$ sufficiently small and focus on the term $f\circ T_\delta^{-1}$. We compute the {first-order} Taylor expansion of $t\mapsto \tilde \LG_t f:=f\circ (tT_{\delta}^{-1}+(1-t)T_0^{-1})$ for  $f\in C^\infty$: 
\begin{align}
    \tilde \LG_1 f(\cdot)=f\circ T_0^{-1}(\cdot) + (D_{T_0^{-1}(\cdot)}f)(T_\delta^{-1}(\cdot)-T_0^{-1}(\cdot)) + G(\cdot).\label{eq:taylorexp}
\end{align}
with $G(\cdot)$ the error term in the Taylor expansion which can be written as follows,
\begin{align}
G(\cdot):&=\int_0^1 \frac{(1-s)^2}{2!}\left[D^2_{sT_\delta^{-1}(\cdot)+(1-s)T_0^{-1}(\cdot)}f ((T_\delta^{-1}-T_0^{-1})(\cdot))\right](T_\delta^{-1}-T_0^{-1})(\cdot))ds \nonumber\\
&=\sum_{1\leq i, j\leq 2}\frac{1}{2}\int_0^1 (1-s)^2 \tilde \LG_s\left(\partial_i\partial_jf\right) (T_\delta^{-1}(\cdot)-T_0^{-1}(\cdot))_i(T_\delta^{-1}(\cdot)-T_0^{-1}(\cdot))_jds    \label{eq:taylorint2}
\end{align}
One can check that the terms in \eqref{eq:taylorexp} are in $\B^{0,3}$
{using the same reasoning as for $G_0$ in step 3 and the fact that by Proposition \ref{lemmulti} item iv) $\partial_i\partial_j f\in \B^{0,3}$ whenever $f\in \B^{2,1}$ (and here we have $f\in C^\infty\subset \B^{2,1}$).}
We now pass the norm $\|\cdot\|_{\B^{p,q}}$  into the integral of \eqref{eq:taylorint2}. Thus, using item iv) from Proposition \ref{lemmulti},
\begin{align*}
   \left\|G\right\|_{0,3}&\leq C\sum_{1\leq i,j\leq 2}\frac{1}{2}\int_0^1 (1-s)^2 \|\tilde \LG_s\left(\partial_i\partial_jf\right)\|_{0,3}\| (T_\delta^{-1}-T_0^{-1})_i(T_\delta^{-1}-T_0^{-1})_j\|_{C^3}ds\nonumber\\
    &\leq C\sum_{1\leq i,j\leq 2}\frac{1}{2} \| (T_\delta^{-1}-T_0^{-1})_i(T_\delta^{-1}-T_0^{-1})_j\|_{C^3} \|f\|_{2,1}\int_0^1 (1-s)^2ds.
\end{align*}

We passed to the last line of the equation 
using the continuity of $\tilde \LG_t : \B^{0,3} \mapsto \B^{0,3}$ (referring again to \cite{GL06} Lemma 7.2).
Recalling that $T_\delta=T_0+\delta \Td +o_{C^5}(\delta)$, and that $C^4\ni T \mapsto T^{-1}\in C^3$ is {$C^1$} around $T_0$ we obtain
\begin{align*}
    \| (T_\delta^{-1}-T_0^{-1})_i(T_\delta^{-1}-T_0^{-1})_j\|_{C^3}\leq C\delta^2.
\end{align*}
and thus 
\begin{align}\label{eq:gtroi}
    \left\|G\right\|_{0,3}\leq C\delta^2\|f\|_{2,1}.
\end{align}
Since $\mathcal{L}_\delta=\tilde{\mathcal{L}_1}/|\det(D_{T_\delta^{-1}x}T_\delta)|$, by \eqref{eq:taylorexp}, we have that
\begin{align}\label{eq:expandldelta}
    \LG_\delta f(x)=\frac{1}{|\det(D_{T_\delta^{-1}x}T_\delta)|}f\circ T_0^{-1}(x) +\frac{1}{|\det(D_{T_\delta^{-1}x}T_\delta)|}(D_{T_0^{-1}}f)(T_\delta^{-1}(x)-T_0^{-1}(x)) + \frac{G(x)}{|\det(D_{T_\delta^{-1}x}T_\delta)|}.
\end{align}
We recall the following expanded expression of $\LGd$ from \eqref{eq:ldotderiv1},
\begin{align*}
    \partial_{\delta=0} \frac{f\circ T_\delta^{-1}(\cdot)}{|\det(D_{T_\delta^{-1}(\cdot)}T_{\delta})|}=\frac{1}{|\det(D_{T_0^{-1}(\cdot)}T_{0})|}\partial_\delta \left(f\circ T_\delta^{-1}(\cdot)\right)+f\circ T_0^{-1}(\cdot)\partial_{\delta=0} \frac{1}{|\det(D_{T_\delta^{-1}(\cdot)}T_{\delta})|},
\end{align*}
By combining this with the expansion  \eqref{eq:expandldelta} we obtain the following expression,
\begin{align}
    &\LG_\delta f-\LG_0 f- \delta \dot{\LG}f\nonumber \\
=&\frac{1}{|\det(D_{T_\delta^{-1}(\cdot)}T_{\delta})|}f\circ T_0^{-1}-\LG_0 f-\delta f\circ T_0^{-1}(\cdot)\partial_{\delta=0} \frac{1}{|\det(D_{T_\delta^{-1}(\cdot)}T_{\delta})|}\label{eqh1}\\
&+\frac{1}{|\det(D_{T_\delta^{-1}(\cdot)}T_\delta)|}(D_{T_0^{-1}(\cdot)}f)(T_\delta^{-1}(\cdot)-T_0^{-1}(\cdot)) - \delta \frac{1}{|\det(D_{T_\delta^{-1}(\cdot)}T_{\delta})|}\partial_\delta \left(f\circ T_\delta^{-1}(\cdot)\right)\label{eqh2}\\
&+\delta \frac{1}{|\det(D_{T_\delta^{-1}(\cdot)}T_{\delta})|}\partial_\delta \left(f\circ T_\delta^{-1}(\cdot)\right)- \delta \frac{1}{|\det(D_{T_0^{-1}(\cdot)}T_{0})|}\partial_\delta \left(f\circ T_\delta^{-1}(\cdot)\right)\label{eqh3}\\
&+\frac{1}{|\det(D_{T_\delta^{-1}(\cdot)}T_\delta)|}G.\label{eqh4}
\end{align}

Taking $H_1,H_2,H_3$ and $H_4$ as the respective terms \eqref{eqh1}, \eqref{eqh2}, \eqref{eqh3} and \eqref{eqh4} we will write
\begin{align*}
\left\|\LG_\delta f-\LG_0 f- \delta \dot{\LG}f \right\|_{0,3}&\leq H_1+H_2+H_3 +H_4.
\end{align*}

We then control $H_1$, $H_2$, $H_3$, and $H_4$ using
the multiplicativity of the norms (Proposition \ref{lemmulti} item v)) and the fact that $\|\cdot\|_{\B^{p,q}}\leq \|\cdot\|_{\B^{p+1,q-1}}$: we start with $H_1$ using relation \eqref{eqdifftmoins2},
\begin{align*}
    H_1&:=\|(f\circ T_0^{-1})\cdot \left(\frac{1}{|\det(D_{T_\delta^{-1}(\cdot)}T_\delta)|}-\frac{1}{|\det(D_{T_0^{-1}(\cdot)}T_0)|}\right)-\delta\cdot(f\circ T_0^{-1})\cdot \partial_\delta \frac{1}{|\det(D_{T_\delta^{-1}(\cdot)}T_\delta)|}\|_{0,3}\\
    &\leq C\|f\circ T_0^{-1}\|_{0,3}\cdot\|\left(\frac{1}{|\det(D_{T_\delta^{-1}(\cdot)}T_\delta)|}-\frac{1}{|\det(D_{T_0^{-1}(\cdot)}T_0)|}\right)-\delta\cdot \partial_\delta \frac{1}{|\det(D_{T_\delta^{-1}(\cdot)}T_\delta)|}\|_{C^3}\\
    &\le o(\delta)\|f\|_{2,1}.
\end{align*}
For $H_2$ we use the fact that $\|\frac{1}{|\det(D_{T_\delta^{-1}(\cdot)}T_\delta)|}\|_{C^3}$ is uniformly bounded for $\delta$ in a neighbourhood of $\delta=0$ 
\begin{align*}
    H_2&:=\|\frac{1}{|\det(D_{T_\delta^{-1}(\cdot)}T_\delta)|}\left( D_{T_0^{-1}(\cdot)}f (T_\delta^{-1}(\cdot)-T_0^{-1}(\cdot)) -\delta\partial_{\delta=0} \left(f\circ T_\delta^{-1}(\cdot)\right) \right)\|_{0,3}\\
    &\leq  \|\frac{1}{|\det(D_{T_\delta^{-1}(\cdot)}T_\delta)|}\|_{C^3}\|D_{T_0^{-1}(\cdot)}f\|_{0,3}\| (T_\delta^{-1}(\cdot)-T_0^{-1}(\cdot)) -\delta \partial_{\delta=0} T_\delta^{-1}(\cdot) \|_{C^3}\\
    &\le \|f\|_{1,2}o(\delta)\le \|f\|_{2,1}o(\delta),
\end{align*}
where the last line is a consequence of the control \eqref{eqdifftmoins1}.
As for $H_3$,
\begin{align}
H_3&:=\|\delta\left(\frac{1}{|\det(D_{T_\delta^{-1}(\cdot)}T_\delta)|}-\frac{1}{|\det(D_{(\cdot)}T_0)|}\right)\partial_{\delta=0}(f\circ T_\delta^{-1})\|_{0,3}\nonumber\\
    &\leq 2\delta \left\|\frac{1}{|\det(D_{T_\delta^{-1}(\cdot)}T_\delta)|}-\frac{1}{|\det(D_{(\cdot)}T_0)|} \right\|_{C^3} \|\partial_{\delta=0}(f\circ T_\delta^{-1})\|_{0,3}\nonumber\\
    &\leq 2\delta O(\delta) \|D_xf (\partial_{\delta=0}T_\delta^{-1})\|_{0,3}\nonumber\\
        &\leq O(\delta^2) \|f\|_{1,2} \|\partial_{\delta=0}T_\delta^{-1}\|_{C^3}\label{eqhtroi}\\
        &\leq \|f\|_{2,1}O(\delta^2),\nonumber
\end{align}
where we obtained line \eqref{eqhtroi} using the fact that the map $\tilde \Psi : T\in C^4\mapsto T^{-1}\in C^3$ is $C^1$ and item iv) from Proposition \ref{lemmulti} on the decomposition $D_xf (\partial_{\delta=0}T_\delta^{-1})=\partial_1f(\partial_{\delta=0}T_\delta^{-1})_1+\partial_2f(\partial_{\delta=0}T_\delta^{-1})_2$. 
The last term $H_4$, is controlled thanks to inequality \eqref{eq:gtroi},
\begin{align*}
    H_4:=\|G(\cdot)\frac{1}{|\det(D_{T_\delta^{-1}(\cdot)}T_\delta)|}\|_{\B_{0,3}}\leq O(\delta^2)\|f\|_{2,1}\|\frac{1}{|\det(D_{T_\delta^{-1}(\cdot)}T_\delta)|}\|_{C^3},
\end{align*}
We thus obtain the formula \eqref{eq:Ldotest} for all $f\in C^\infty$. 
Then by continuity of the operators $\LGd,\LG_\delta$ and $\LG_0$ on $\B^{2,1}$ and the density of $C^\infty$ in $\B^{2,1}$, inequality \eqref{eq:Ldotest} also extends to all $f\in \B^{2,1}$
\begin{align*}
   \left\|\LG_\delta f-\LG_0 f- \delta \dot{\LG}f \right\|_{0,3}=o(\delta)\|f\|_{2,1}.
\end{align*}

\end{proof}

\subsection{Existence of linear response}

{For $T_\delta$ close enough to $T_0$ in $C^{r+1}$, one has} $T_\delta$ is topologically transitive and by Theorem 2.3 \cite{GL06} and the discussion at the start of Section 7 \cite{GL06}, $T_\delta$ has a unique SRB measure $f_\delta\in \mathcal{B}^{2,1}$. 
Thus one can ask whether there is a linear response $R:C^{\rr+1}(\mathbb{T}^2,\mathbb{R}^2)\to \mathcal{B}^{0,3}$ of the system
to a perturbation $\dot{T}\in C^{\rr+1}(\mathbb{T}^2,\mathbb{R}^2)$, defined as
\begin{align}\label{eqdefrtd}
R(\dot{T}):=\lim_{\delta \rightarrow 0}\frac{f_{\delta }-f_{0}}{\delta}.
\end{align}
While the above convergence takes place in $\B^{0,3}$, we will show in Theorem \ref{thmresp} that in fact, the limit point lies in $\B^{1,2}$. 
When we write $R(\dot{T})\in \B^{1,2}$ in Theorem \ref{thmresp} this is what is meant. 
Before stating Theorem \ref{thmresp}, notice that by definition $\LG_\delta f_\delta=f_\delta$, thus one can guess some heuristic link between $R(\Td)$ and the derivative operator $\LGd$ as follows,
\begin{align}
    R(\dot{T})&:=\lim_{\delta \rightarrow 0}\frac{f_{\delta }-f_{0}}{\delta}\nonumber\\
    &=\lim_{\delta \rightarrow 0}\frac{\LG_\delta f_{\delta }-\LG_0 f_{0}}{\delta}\nonumber\\
    &=\lim_{\delta \rightarrow 0}\LG_\delta\frac{f_{\delta }-f_{0}}{\delta}+\lim_{\delta \rightarrow 0}\frac{\LG_\delta-\LG_{0}}{\delta}f_0\label{eqrtheuristic}\\
    &=\LG_0 R(\Td)+\LGd f_0\nonumber,
\end{align}
which suggests $R(\dot T)=(I-\mathcal{L}_0)^{-1}\dot{\mathcal{L}}(\dot{T})(f_0)$. To formalise this link one has to study the existence and continuity properties of operator $(I-\mathcal{L}_0)^{-1}$. For that purpose we introduce the Banach spaces $V_{p,q}$ of centered distributions, for $p\in \NN$ and $q>0$,
$$
V_{p,q}:=\{f \in \B^{p,q}: \int f =0\}.
$$
{The following theorem is an application of Theorem 2 \cite{KL99} to our setting.}
\begin{theorem}[Theorem 2 \cite{KL99}]\label{thm:resolv}
   Assume {there is a $\bar\delta>0$ such that}
\begin{itemize}
\item[1)] there is $M>0$ and {$C_1>0$} such that {for all $\delta\le \bar\delta$ and all $n\ge 0$} one has $$\|\LG_\delta^n f\|_{0,3}\leq C_1 M^n\|f\|_{0,3};$$
\item[2)] there is $0<\alpha <M$, {$C_2, C_3>0$} such that for $\delta\leq \bar \delta$, every $n\ge 0$, {and every $f\in \mathcal{B}_{1,2}$}
$$
\|\LG_\delta^n f\|_{1,2}\leq C_2\alpha^n \|f\|_{1,2}+C_3M^n \|f\|_{0,3};
$$

\item[3)] there is a monotone increasing
map $\tau : (0,\bar \delta]\to[0,\infty)$ {with} $\lim_{\delta \to 0} \tau(\delta)=0$ {such that for $\delta\le \bar\delta$ one has}
$$
\|\LG_\delta-\LG_0\|_{\B^{1,2}\to \B^{0,3}}\leq \tau(\delta);
$$
\end{itemize}
Then 
for any $\rho>\alpha$, {$\delta\leq \bar\delta$, and for any $z\in \mathbb{C}$ such that $|z|>\rho$ and $d(z,sp(\LG_0))>0$ (where $sp(\LG_0)$ is the spectrum of $\LG_0$ as an operator on
$\B^{1,2}$), one has}
{
\begin{enumerate}
    \item 
the map $\delta \mapsto (z-\LG_\delta)^{-1}$ is continuous.
\item there is $\eta_\rho>0$ and $C=C(z)$ such that
  \begin{align}\label{eq:resolv2}
       \|(z-\LG_\delta)^{-1}-(z-\LG_0)^{-1}\|_{V_{1,2}\rightarrow V_{0,3}}\leq C|\tau(\delta)|^{\eta_\rho}.
   \end{align}
\end{enumerate}}
\end{theorem}

Using this result we are now ready to formally state an expression of $R(\Td)$ in terms of $\LGd: \B^{2,1}\to \B^{1,2}$.
The expression {for $R(\dot{T})$ in Theorem \ref{thmresp} is well known (see e.g.\ Theorem 2 in \cite{Po19}),
but thanks to Theorem \ref{prop:taylorcon}, it only requires $\delta \mapsto T_\delta$ to be $C^1$}. For the sake of completeness we also give the proof.

\begin{theorem}\label{thmresp}
The map $\Td\in C^{\rr+1}(\mathbb{T}^2,\mathbb{R}^2)\mapsto R(\Td) \in V_{1,2}\subset\B^{1,2}$ with $R(\Td)$ given by \eqref{eqdefrtd} is well defined and can be expressed as
\begin{align}
\label{eq:response}
R(\dot T)=(I-\mathcal{L}_0)^{-1}\dot{\mathcal{L}}(\dot{T})(f_0),    
\end{align}
with $\dot{\mathcal{L}} : \B^{2,1} \mapsto \B^{1,2}$ 
as defined in Theorem \ref{prop:taylorcon}.
\end{theorem}

\begin{proof}

Let us explain why the operator $(I-\mathcal{L}_0)^{-1}$ appearing in the expression of $R(\Td)$ is well defined.
{Theorem \ref{prop:taylorcon} states that $\dot{\mathcal{L}}:\mathcal{B}^{2,1}\to\mathcal{B}^{1,2}$. Moreover,} for any $f\in \B^{2,1}$, $\LGd f=\lim_{\delta \to 0} \frac{\LG_\delta- \LG_0}{\delta}f$ and by definition of the transfer operator $\LG_\delta$, $\frac{\LG_\delta- \LG_0}{\delta}f \in V_{1,2}$, thus $\LGd f \in V_{1,2}$.
{Thus, in the expression \eqref{eq:response}, the operator $(I-\LG_0)^{-1}$ acts on elements of $V_{1,2}$, and so} $(I-\LG_0)^{-1}:= \sum_{k\geq 0}\LG_0^k$ is well defined on $V_{1,2}$ {because of the spectral gap of $\LG_0$ acting on $\mathcal{B}^{1,2}$}. Indeed according to Theorem 2.3 in \cite{GL06}, there is $0<\sigma<1$ such that for $f\in V_{1,2}$, 
$
\|\LG_0^nf\|_{1,2}\leq \sigma^n\| f\|_{1,2}.
$
Thus $(I-\LG_0)^{-1} : V_{1,2}\mapsto V_{1,2}$ is a continuous linear operator. Recall from Theorem 2.3  \cite{GL06}, $f_0\in B^{2,1}$ and
following the heuristic reasoning \eqref{eqrtheuristic}, we have that 
\begin{align*}
    \frac{f_\delta-f_0}{\delta}=\LG_\delta \frac{f_\delta-f_0}{\delta}+\frac{\LG_\delta-\LG_0}{\delta}f_0,
\end{align*}
and deduce that since $\frac{\LG_\delta-\LG_0}{\delta}f_0 \in V_{1,2}$,
\begin{align*}
    \frac{f_\delta-f_0}{\delta}=(I-\LG_\delta)^{-1}\frac{\LG_\delta-\LG_0}{\delta}f_0.
\end{align*}
{We now begin our main computation:}
\begin{align}
\nonumber \|R(\dot{T}) - ((I-\mathcal{L}_0)^{-1}\dot{\mathcal{L}}(\dot T)(f_0)\|_{0,3}&\leq \left\|R(\dot T)- \frac{f_\delta -f_0}\delta\right\|_{0,3}+\left\| \frac{f_\delta-f_0}{\delta}-(I-\LG_0)^{-1}\frac{\LG_\delta-\LG_0}{\delta}f_0\right\|_{{0,3}}\\
 \label{eqlinresp}&+  \|((I-\mathcal{L}_0)^{-1}\dot{\mathcal{L}}(\dot T)(f_0) - ((I-\mathcal{L}_0)^{-1}(({\mathcal{L}_\delta-\mathcal{L}_0})/\delta)(\dot T)(f_0)\|_{0,3}
\end{align}
In the above inequality, the term $\|R(\dot T)- (f_\delta -f_0)/\delta\|_{0,3}$ goes to zero by definition of $R(\Td)$. As for the latter term in \eqref{eqlinresp},
\begin{align*}
    \left\|((I-\mathcal{L}_0)^{-1}\left(\dot{\mathcal{L}}(\dot T)(f_0) - \frac{\mathcal{L}_\delta-\mathcal{L}_0}\delta\right)(f_0)\right\|_{0,3}\leq \|(I-\mathcal{L}_0)^{-1}\|_{1,2}\left\|\frac{\mathcal{L}_\delta-\mathcal{L}_0}\delta(f_0)-\LGd(\Td)f_0\right\|_{1,2}.
\end{align*}

Notice that according to {equation \eqref{eqmainthmcondelta}} in Theorem \ref{prop:taylorcon} {with $p=3, q=0$}, $\|\frac{\mathcal{L}_\delta f_0-\mathcal{L}_0f_0}\delta\|_{2,1}$ is bounded, thus the family $\{(\mathcal{L}_\delta f_0-\mathcal{L}_0f_0)/\delta\}$ is relatively compact in $\B^{1,2}$ (compact embedding, item ii) from Proposition \ref{lemmulti}) and according to Theorem 3.1, the only possible limit is the limit in $B^{0,3}$, namely $\LGd(\Td)f_0$. Thus the last term in \eqref{eqlinresp} goes to $0$. 
 Now to control the remaining term in equation \eqref{eqlinresp}, notice that on $V_{1,1}$ the spectrum of $\LG_0$ is bounded away from $1$ (all the spectral value are in a disk of radius $\alpha$ around $0$), thus \eqref{eq:resolv2} from
Theorem \ref{thm:resolv} applies for $\alpha <\rho<1$. Then according to equation \eqref{eqmainthmcondelta} from Theorem \ref{prop:taylorcon},
\begin{align*}
    \left\| \frac{f_\delta-f_0}{\delta}-(I-\LG_0)^{-1}\frac{\LG_\delta-\LG_0}{\delta}f_0\right\|_{{0,3}}&\leq \left\|(I-\LG_\delta)^{-1}-(I-\LG_0)^{-1}\right\|_{V_{1,2}\rightarrow V_{0,3}}\left\|\frac{\LG_\delta-\LG_0}{\delta}f_0\right\|_{{1,2}}\nonumber\\
    &\leq C\delta^{\eta_\rho}.
\end{align*}
Then, collecting everything in inequality \eqref{eqlinresp} yields,

\begin{align*}
    R(\Td)&:=\lim_{\delta \to 0}\frac{f_\delta-f_0}{\delta} \\
    &=(I-\LG_0)^{-1}\LGd f_0.
\end{align*}

{Notice that although the above limit in $\delta$ is taken in $\B^{0,3}$, one may show that
in fact $(I-\LG_0)^{-1}\LGd f_0 \in \B^{1,2}$ as follows.
Recall the formula for $\LGd$ given by equation \eqref{eq:divrep}.
Since $f_0\in \B^{2,1}$ and $\Td,T_0\in C^{r+1}$, according to item iv) and v) from Proposition \ref{lemmulti} one has
\begin{align*}
    \nabla\cdot (f_0 (D T_0)^{-1}(\dot T)) \in \B^{1,2}.
\end{align*}
Thus $ \LGd f_0:=-\LG_0\left(\nabla \cdot \left(f_0(D T_0)^{-1}(\dot T)\right) \right)\in \B^{1,2}$ and since $(I-\LG_0)^{-1} :V^{1,2}\mapsto V^{1,2}$ and $\LGd f_0\in V^{1,2}$, as stated at the beginning of this proof,
\begin{align*}
(I-\LG_0)^{-1}\LGd f_0\in V^{1,2}.
\end{align*}}
\end{proof}

\section{Optimising linear response of the expectation of observations}\label{secjt}

We wish to select a perturbation ${\dot{T}\in C^{r+1}}$ that maximises the
rate of change of the expectation of a chosen observable $c\in C^2(\mathbb{T}^2,\mathbb{R})$.
For the purpose of optimization, it is convenient to restrict the perturbations $\dot{T}$ to the Sobolev space $\mathcal{H}:=H^7(\mathbb{T}^2,\mathbb{R}^2)$, {since by}
 a Sobolev embedding theorem (see e.g.\ Corollary 9.15 from \cite{brezis}), $H^7(\mathbb{T}^2)\subset {C^{\rr+1}}(\mathbb{T}^2)$.
We define an objective function $J:\mathcal{H}\to\mathbb{R}$ by $J(\dot{T})=\int c\cdot R(\dot{T})$, which records the expectation of the observation $c$.

\subsection{Continuity of the objective $J$}
\label{sec:Jcontinuity}
The following theorem establishes continuity of the functions $R$ and $J$. Notice that $\LGd(\Td)$ and $R(\Td)$ given respectively by the formulae \eqref{eq:divexp} and   \eqref{eq:response}, are well defined as soon as $\Td\in C^5$.

\begin{theorem}
\label{thm:deriv}
{For $T_0\in C^5$,}

\begin{enumerate}
\item The map $\Td\in C^{5}\mapsto \dot{\mathcal{L}}(\Td) \in L(\B^{2,1},\B^{1,2})$ is continuous.
\item The map $\Td\in C^{{5}}\mapsto R(\Td) \in V_{1,2}$ is continuous.
\item  for $c\in C^{3}(\TT,\RR)$, the map $J : C^{5} \mapsto \RR$ defined by $J(\Td):= \int c R(\Td)$ is continuous.
\end{enumerate}
\end{theorem}

\begin{proof}
We prove each item in order as each item implies the next.
\paragraph{Item 1.} 
It is sufficient to prove that there is $C>0$ such that 
\begin{align*}
    \|\dot{\mathcal{L}}(\dot{T})(f_0)\|_{\mathcal{B}_{1,2}}\leq C\|\Td\|_{C^5}.
\end{align*}

According to expression \eqref{eq:divexp} from the proof of Theorem \ref{prop:taylorcon}, $\LGd f_0$ can be written as 
\begin{align*}
    \LGd f_0(x)&=-\LG_0\left(\Tr\left(D_x f_0 \circ(D_x T_0)^{-1}(\dot T(x))\right)\right)-\LG_0\left( f_0\Phi\right)(x)\\
    &=-\LG_0\left(D_x f_0 \circ(D_x T_0)^{-1}(\dot T(x))\right)-\LG_0\left( f_0\Phi\right)(x).
\end{align*}
with $\Phi \in  C^3(\TT,\mathbb{R})$ given by
 $\Phi(x)=\langle \nabla\cdot (D_xT_0)^{-1},\dot T(x)\rangle+ \Tr((D_xT_0)^{-1}\circ D_x\dot T).$
One can rewrite the map $x\in \TT \mapsto (D_x T_0)^{-1}(\dot T(x))$ by separating coordinates via\\
$(D_x T_0)^{-1}(\dot T(x)):=(\phi_1(x),\phi_2(x))$ to yield 
\begin{align}
\label{66}
    \LGd f_0=-\LG_0\left(\partial_1 (f_0) \phi_1 +\partial_2 (f_0) \phi_2\right)-\LG_0 (f_0\cdot\Phi).
\end{align}

According  to
Theorem 2.3 \cite{GL06}, $f_0 \in \B^{2,1} \subset \B^{1,2}$. 
The multiplicativity property (item v) from Proposition \ref{lemmulti}) and the continuity of $f\in \B^{2,1}\mapsto \partial_i f\in \B^{1,2}$ (item iv) of Proposition \ref{lemmulti}) applied to \eqref{60} and \eqref{66} yields
\begin{align}
 \|\LGd(\Td)\|_{{1,2}} &\leq \|\LG_0\|_{1,2}\|f_0\Phi+  \partial_1 (f_0) \phi_1 +\partial_2 (f_0) \phi_2\|_{1,2}\nonumber\\
 &\leq \|\LG_0\|_{1,2}\left(C_1\|f_0\|_{1,2}\|\Phi\|_{C^{3}} +C_2\|\partial_1 f_0\|_{1,2}\|\phi_1\|_{C^{3}}+\|\partial_2 f_0\|_{1,2} \|\phi_2\|_{C^{3}}\right)\nonumber\\
 &\leq \|\LG_0\|_{1,2}\left(C_1\|f_0\|_{\p,\q}\|\Phi\|_{C^{3}} +C_2\|f_0\|_{\p,\q}\|\phi_1\|_{C^{3}}+\|f_0\|_{2,1} \|\phi_2\|_{C^{3}}\right),\label{eq:conrt}
\end{align}
where the last line is obtained noticing that by definition of the norm on the Banach space $\B^{p,q}$, for any $f\in \B$, $\|\partial_i f\|_{\B^{1,2}}\leq \| f\|_{\B^{2,1}}$.
Finally, since $ \Td \in (C^{5}(\TT,\mathbb{R}^2),\|\cdot\|_{C^5}) \mapsto \Phi \in (C^3(\TT,\mathbb{R}),\|\cdot\|_{C^3}) $ and $ \Td \in (C^{5}(\TT,\mathbb{R}^2),\|\cdot\|_{C^5}) \mapsto (\phi_1,\phi_2 )\in (C^3(\TT,\mathbb{R}^2),\|\cdot\|_{C^3})$ are {linear} maps, {continuous} in $\Td$ and $D_{\cdot}\Td$ respectively, there is a constant $C>0$ such that for any $i\in\{1,2\}$
\begin{align*}
    \|\phi_i\|_{C^3}&\leq C\|\Td\|_{C^5}\\
    \|\Phi\|_{C^3}&\leq C\|\Td\|_{C^5}
\end{align*}

And thus introducing these relations in the formula \eqref{eq:conrt}, we obtain
\begin{align}
   \label{Ldotbound} \|\LGd(\Td)\|_{1,2}\leq C\|\Td\|_{C^5}
\end{align}

\paragraph{Item 2.}
{According to Theorem \ref{thmresp} and the expression of $\LGd$ from Theorem \ref{prop:taylorcon}, $R(\cdot)$ is a linear operator {from} $C^{5}(\TT,\mathbb{R}^2)$ {to $V_{1,2}\subset\mathcal{B}^{1,2}$}.
Indeed, we recall the expression of $R(\Td)$ from equation \eqref{eq:response},}
\begin{align}
\label{60}
R(\dot T)(\cdot)=\left((I-\mathcal{L}_0)^{-1}\dot{\mathcal{L}}(\dot{T})(f_0)\right)(\cdot),  
\end{align}
with $(I-\mathcal{L}_0)^{-1}$ a continuous operator on $V_{1,2}$ {(as shown at the beginning of the proof of Theorem \ref{thmresp}) }
independent of the choice of $\Td$
 and  $\LGd(\Td)f_0\in  V^{1,2}$ as stated at the beginning of the proof of Theorem \ref{thmresp}.
From the expression \eqref{60} for $R(\Td)$, and using \eqref{Ldotbound} one has 
\begin{align*}
 \|(I-\LG_0)^{-1}\LGd(\Td)f_0\|_{{1,2}} &\leq \|(I-\LG_0)^{-1}\|_{V_{1,2} \to V_{1,2}}\|\LGd(\Td)f_0\|_{1,2}\nonumber\\
 &\leq C \|(I-\LG_0)^{-1}\|_{V_{1,2} \to V_{1,2}}\|\LGd(\Td)f_0\|_{1,2}\\
 &\leq C'\|\Td\|_{C^5}.
\end{align*}
\paragraph{Item 3.}
{Since $J$ is a linear operator on $\Hc$, to prove its continuity, it is enough to prove the continuity of $\Td \mapsto R(\Td)\in \B^{1,2}$. }
Indeed, then according to the display equation above Proposition 4.1 in \cite{GL06} we have
\begin{equation}
    \label{Jcty}
|J(\Td)|\leq C|c|_{C^2} \|R(\Td)\|_{1,2}.
\end{equation}
Thus the continuity of $J(\cdot)$ follows from the continuity of $\Td\mapsto R(\Td)$ proven in Item 2.

\end{proof}

\subsection{Uniqueness of the optimal perturbation}
\label{sec:uniquepert}

We constrain the infinitesimal perturbations $\dot{T}$ to lie in a feasible set $P\subset\mathcal{H}$ that we assume is bounded and strictly convex.
Boundedness of additive perturbations is a natural assumption, as is convexity of the class of perturbations because a perturbation arising as the interpolation of two allowed perturbations ought to also be allowed.
The following abstract result yields the existence of a unique optimal perturbation for non-degenerate functionals $J$.
\begin{prop}\label{propoptigen}
    Suppose $P$ is closed, bounded, and
strictly convex\footnote{We say that a convex closed set $P\subseteq \mathcal{H}$ is \textit{strictly convex} if for each pair
$x, y \in P$ and for all $0 < s < 1$, the points $s x+(1-s)y \in \mathrm{int}(P)$, where the relative interior
is meant.} subset of a separable Hilbert space $\mathcal{H}$, and that $P$ contains the zero vector in its relative interior. 
If $J$ is a continuous linear function that
does not uniformly vanish on $P$ then the optimal solution to 
\begin{equation*}
J(p^*) = \max_{p\in P}
J(p).
\end{equation*}
is unique.
\end{prop}
See e.g.\ Lemma 6.2 \cite{FKS20} for a proof.
We will apply this result with $P$ being the unit ball in $H^7(\mathbb{T}^2,\mathbb{R}^2)$ and $J(\dot{T})$ being the expected response of an observation $c$ under a perturbation $\dot{T}\in P= B_{H^7}(0)$.
Theorem \ref{thm:deriv} provides continuity of $J:C^{\rr+1}\to\mathbb{R}$, and 
since $H^7\subset C^{\rr+1}$, by Proposition \ref{propoptigen} the following optimisation problem has a unique solution:

\begin{eqnarray}
\label{optobs}
\max_{\dot{T}\in H^7}&&{J}(\dot{T})\\
\label{optobs2}
\mbox{subject to}&&\|\dot{T}\|_{H^7}\le 1.
\end{eqnarray}

{ 
We will see in the next section that in the particular case where the convex set $P$ is the unit ball $B_{H^7}(0)$, the map $J(\cdot)$ attains its maximum  on the ball $B_{H^7}(0,1)$ and that maximum $\tilde T\in H^7(\TT,\mathbb{R}^2)$ is given by $\tilde T=\frac{u}{\|u\|_{H^7}}$ where $u\in H^7$ is the element such that $J(\cdot)=\langle \cdot,u\rangle_{H^7}$ in the Riesz representation Theorem. 
}

\subsection{An explicit formula for the optimal perturbation}
\label{sec:opt}

We now state a theorem identifying the optimal perturbation for the problem \eqref{optobs}. 
Although we optimize $J(\cdot)$ over $\Td\in H^7(\TT,\mathbb{R}^2)$, for numerical purposes we use complex Fourier modes to approximate the optimal $\dot{T}$.
We therefore embed $H^7(\TT,\mathbb{R}^2)$ in $H^7(\TT,\mathbb{C}^2)$
with inner product $\langle\cdot,\cdot\rangle_{H^7}$ defined by $\langle u,v\rangle_{H^7}=\sum_{m\leq 7}\int_{\mathbb{T}^2} D^mu\overline{D^mv}$. 
By a slight abuse of notation, we also denote the complex extension of the objective function by $J(\cdot)$.

For $\mathbf{k}\in\mathbb{Z}^2$, let the elementary scalar-valued Fourier modes $\exp(2\pi i \mathbf{k}\cdot x)$ on $\mathbb{T}^2$ be denoted $e_{\mathbf{k}}:\mathbb{T}^2\to\mathbb{C}$.
In the following theorem we construct the optimal perturbation $\dot{T}:\mathbb{T}^2\to\mathbb{C}^2$ with the ansatz
\begin{equation}
 \label{dotTeqn} \dot{T}=\left(\sum_{\mathbf{k}\in\mathbb{Z}^2} a^{(1)}_{\mathbf{k}}e_{\mathbf{k}},\sum_{\mathbf{l}\in\mathbb{Z}^2} 
 a^{(2)}_{\mathbf{l}}e_{\mathbf{l}}\right),
 \end{equation}
where 
$a^{(1)}_{\mathbf{k}}$ and $a^{(2)}_{\mathbf{l}}$ are complex scalar coefficients.
\begin{theorem}
\label{mainthm1}
Let $\{T_\delta\}_{0\le\delta<\delta_0}$ be a family of $C^{\rr+1}$ Anosov maps $T_\delta:\mathbb{T}^2\to\mathbb{T}^2$ satisfying \eqref{Ttaylor}, and $c\in C^3(\mathbb{T}^2,\mathbb{R})$ an observation function chosen so that $J$ does not uniformly vanish.
The perturbation $\dot{T}\in H^7(\mathbb{T}^2,\mathbb{R}^2)$ that maximises the expected linear response $\int_{\mathbb{T}^2} c\cdot R(\dot T)$ (solves the optimisation problem \eqref{optobs}--\eqref{optobs2}) is given by
\eqref{dotTeqn} with coefficients  
\begin{equation}
    \label{akneqnthm}
a_{\mathbf{k}}^{(1)}=-\int c\cdot \overline{(I-\mathcal{L}_0)^{-1} \mathcal{L}_0\left(\nabla\cdot\left(  f_0(\cdot)(D_\cdot T_0)^{-1}(e_{\mathbf{k}},0)(\cdot))\right)\right)} \left. \middle/ \nu  \left(\sum_{m=0}^7  (2\pi)^{2m} |\mathbf{k}|^{2m}_2\right)\right.,
and 
\end{equation}
\begin{equation}
\label{alneqnthm}
a_{\mathbf{l}}^{(2)}=-\int c\cdot \overline{(I-\mathcal{L}_0)^{-1} \mathcal{L}_0\left(\nabla\cdot \left( f_0(\cdot)(D_\cdot T_0)^{-1}(0,e_{\mathbf{l}})(\cdot))\right)\right)} \left. \middle/ \nu  \left(\sum_{m=0}^7  (2\pi)^{2m} |\mathbf{l}|^{2m}_2\right)\right.
\end{equation}
where $\nu>0$ is chosen so that $\|\dot{T}\|_{H^7}=1$;  that is $$\sum_{\mathbf{k}\in\mathbb{Z}^2}\left(a_{\mathbf{k}}^{(1)}\right)^2\sum_{m=0}^7 \left(2\pi|\mathbf{k}|_2\right)^{2m}+\sum_{\mathbf{l}\in\mathbb{Z}^2}\left(a_{\mathbf{l}}^{(2)}\right)^2\sum_{m=0}^7 \left(2\pi|\mathbf{l}|_2\right)^{2m}=1.$$

\end{theorem}
\begin{proof} 
{When $J(\cdot)$ is non-degenerate on the Hilbert space $H^7:=H^7(\TT,\mathbb{C}^2)$, the Riesz representation theorem ensures the existence of a unique $u\in H^7\backslash \{0\}$ such that for any $\Td\in H^7$,
\begin{align*}
    J(\Td)=\langle \Td,u\rangle_{H^7}.
\end{align*}
We first prove that $u$ is real-valued, i.e.\ $\Im(u)=0$. Since $\langle \cdot,\cdot\rangle_{H^7}$ also represents a scalar product on the subspace $\{v \in H^7: \Im(v)=0\}$, Riesz applied to this subspace guarantees a $u_{\rm real}\in  \{v \in H^7: \Im(v)=0\}$ such that for any real-valued $\Td\in H^7(\TT,\mathbb{R}^2)$, one has  $J(\Td)=\langle \Td,u_{\rm real}\rangle_{H^7}$.
Because any $\Td\in H^7(\TT,\mathbb{C}^2)$ can be written as $\Td=\Re(\Td)+i\Im(\Td)$, 
\begin{align*}
     J(\Td)&=\langle \Td,u\rangle_{H^7}\\
     &=\langle \Re(\Td),u\rangle_{H^7}+i\langle \Im(\Td),u\rangle_{H^7}\\
     &=\langle \Re(\Td),u_{\rm real}\rangle_{H^7}+i\langle \Im(\Td),u_{\rm real}\rangle_{H^7}\\
     &=\langle \Td,u_{\rm real}\rangle_{H^7}
\end{align*}
Thus $u=u_{\rm real}$ and $\Im(u)=0$.
By the Cauchy--Schwarz inequality on $\{v\in H^7: \Im(v)=0\}$, for $\Td \in H^7(\TT,\mathbb{R}^2)$,
\begin{align*}
  J(\Td)=\langle \Td,u\rangle_{H^7}=\Re(\langle \Td,u\rangle_{H^7})\leq    \|\Td\|_{H^7}\|u\|_{H^7},
\end{align*}
}
 and $J(\cdot)$ attains its maximum on the $H^7(\TT,\mathbb{R}^2)$ unit ball at $\tilde T:=\frac{u}{\|u\|_{H^7}}$. Let $a_{\mathbf{k}}^{(1)}$ and $a_{\mathbf{l}}^{(1)}$ be the Fourier coefficients in the Fourier decomposition of $\tilde T$;  that is, $\tilde{T}:=\sum_{k\in \ZZ}a_{\mathbf{k}}^{(1)}(e_{\mathbf{k}},0)+\sum_{l\in \ZZ}a_{\mathbf{l}}^{(1)}(0,e_{\mathbf{l}})$. 
Then 
\begin{align*}
\langle (e_{\mathbf{k}},0),\tilde T\rangle_{H^7}=\overline{a_{\mathbf{k}}^{(1)}}\|(e_{\mathbf{k}},0)\|_{H^7}^2.
\end{align*}

Thus
\begin{equation}
{\overline{a_{\mathbf{k}}^{(1)}}}=\frac{\langle (e_{\mathbf{k}},0),\tilde T\rangle_{H^7}}{\|(e_{\mathbf{k}},0)\|_{H^7}^2}  =\frac{\langle (e_{\mathbf{k}},0),u\rangle_{H^7}}{\|u\|_{H^7}\|(e_{\mathbf{k}},0)\|_{H^7}^2}
=\frac{J((e_{\mathbf{k}},0))}{\|u\|_{H^7}\|(e_{\mathbf{k}},0)\|_{H^7}^2}.
    \label{Jfourier}
\end{equation}

Recall that by definition (see Theorem \ref{thm:deriv}), for $\Td\in H^7$, $J(\Td)=\int c\cdot R(\Td)$. 
Now from \eqref{eq:response} and \eqref{eq:divrep}, one has 
\begin{equation*}
    J(\Td)=\int c\cdot R(\Td)=\int c\cdot (I-\LG_0)^{-1}\LGd(\Td)f_0=-\int c\cdot (I-\LG_0)^{-1}\LG_0\left(\nabla \cdot \left(f_0(D T_0)^{-1}(\dot T)\right)\right).
\end{equation*}

Because $\|(e_{\mathbf{k}},0)\|_{H^7}^2=\sum_{m=0}^7  (2\pi)^{2m} |\mathbf{k}|^{2m}_2$, by \eqref{Jfourier} we arrive at 
\begin{align*}
   a_{\mathbf{k}}^{(1)}= -\int c\overline{\cdot (I-\mathcal{L}_0)^{-1} \mathcal{L}_0\left(\nabla\cdot\left(  f_0(\cdot)(D_\cdot T_0)^{-1}(e_{\mathbf{k}},0)(\cdot))\right)\right)} \left. \middle/ \nu  \left(\sum_{m=0}^7  (2\pi)^{2m} |\mathbf{k}|^{2m}_2\right)\right.,
\end{align*}
where $\nu:=\|u\|_{H^7}$. The same reasoning holds for $a_{\mathbf{l}}^{(2)}$ and the choice of  $\nu$ ensures $\|\tilde T\|_{H^7}=1$. 
\end{proof}

\section{Numerically computing the optimal perturbation $\dot{T}$}
\label{sec:numerics}
This section describes how to compute the vector-valued optimal $\dot{T}$ through the coefficients $a_{\mathbf{k}}^{(1)}, a_{\mathbf{l}}^{(2)}$ in Theorem \ref{mainthm1}.
The main numerical goal is to compute the integrals in \eqref{akneqnthm} and \eqref{alneqnthm}, respectively.
First, we begin with a general formal convergence result, which will later be specialised to our numerical algorithm.

 \subsection{A general convergence result for the optimal perturbation $\dot{T}$}

For suitably small $\epsilon_0>0$, we introduce kernels $k_\epsilon\in
C^\infty(\TT,[0,\infty))$ for ${\epsilon\in (0,\epsilon_0]}$, which satisfy
$\int k_\epsilon\ dm=1$.
According to the proof of Lemma 5.1 \cite{CF20}, the kernel operator $f\in C^\infty(\TT,\mathbb{R})\mapsto k_\epsilon*f \in C^\infty(\TT,\mathbb{R})$ can be extended into a continuous linear operator $K_\epsilon \in L(\B^{p,q})$ for any $p\in \NN$ and $q>0$.

We form approximate transfer operators 
$\mathcal{L}_{\epsilon}=K_\epsilon\circ\mathcal{L}_0$\footnote{In section \ref{sec:numerics}, $\mathcal{L}_\epsilon$ will denote a perturbation of $\mathcal{L}_0$ due to convolution by kernel $k_\epsilon$.  This overloads the previous notation $\mathcal{L}_\delta$, which denoted a perturbation of $\mathcal{L}_0$ due to a perturbation of $T_0$, but as we never consider such $\mathcal{L}_\delta$ in this section, there should be no confusion. A similar overloading occurs with $f_\epsilon$ vs $f_\delta$.}.
{Whenever the family $\{\LG_\epsilon\}_{\epsilon \geq 0}$ satisfies the Keller Liverani (KL) conditions (KL) from \cite{CF20} on $(\|\cdot\|_{p,q},\|\cdot\|_{p-1,q+1})$ for any $p,q\in \NN$ such that $p+q=2$, they consist of quasi-compact operators on $\B^{p,q}$ and for $0<\epsilon\le \epsilon_0$ admit a unique leading eigenvalue with a unique corresponding eigendistribution denoted $f_{\epsilon}$.}
By replacing $\mathcal{L}_0$ with $\mathcal{L}_{\epsilon}$ and $f_0$ with $f_{\epsilon}$, {whenever it is well defined,} in \eqref{akneqnthm} and \eqref{alneqnthm}, we obtain coefficients denoted $a^{(1)}_{\mathbf{k},\epsilon}$ and $a^{(2)}_{\mathbf{l},\epsilon}$:
\begin{equation}
    \label{akneqnthmeps}
a^{(1)}_{\mathbf{k},\epsilon}:=-\int c\cdot \overline{(I-\mathcal{L}_\epsilon)^{-1} \mathcal{L}_\epsilon\left(\nabla\cdot\left(  f_\epsilon(\cdot)(D_\cdot T_0)^{-1}(e_{\mathbf{k}},0)(\cdot))\right)\right)} \left. \middle/ \nu_\epsilon  \left(\sum_{m=0}^5  (2\pi)^{2m} |\mathbf{k}|^{2m}_2\right)\right.,
and 
\end{equation}
\begin{equation}
\label{alneqnthmeps}
a^{(2)}_{\mathbf{l},\epsilon}:=-\int c\cdot \overline{(I-\mathcal{L}_\epsilon)^{-1} \mathcal{L}_\epsilon\left(\nabla\cdot \left( f_\epsilon(\cdot)(D_\cdot T_0)^{-1}(0,e_{\mathbf{l}})(\cdot))\right)\right)} \left. \middle/ \nu_\epsilon  \left(\sum_{m=0}^5  (2\pi)^{2m} |\mathbf{l}|^{2m}_2\right)\right.,
\end{equation}
where $\nu_\epsilon>0$ is chosen to ensure that $\|\dot{T}_\epsilon\|_{H^7}=1$, where $\dot{T}_\epsilon$ is given in \eqref{dotTeqneps}.
We wish to show $\lim_{\epsilon\to 0}a^{(1)}_{\mathbf{k},\epsilon}=a^{(1)}_\mathbf{k}$ and $\lim_{\epsilon\to 0}a^{(2)}_{\mathbf{l},\epsilon}=a^{(2)}_\mathbf{l}$, which will in turn imply convergence of our estimated $\dot{T}_\epsilon$ to $\dot{T}$.
We assume that the map $T_0$ and its derivative are known perfectly, and that the integrals in \eqref{akneqnthm} and \eqref{alneqnthm} are calculated perfectly;  in practice, this is true up to  numerical precision.

\begin{theorem}\label{thmmolli}
Let $c\in C^3(\mathbb{T}^2,\mathbb{R})$ and the kernels $\{k_\epsilon\}_{\epsilon\in (0,\epsilon_0]}$ be chosen such that the family of operator satisfy $(\|\cdot\|_{2,1},\|\cdot\|_{1,2})$ and $(\|\cdot\|_{1,2},\|\cdot\|_{0,3})$ (KL) conditions.
Then the sequence of vector fields \begin{equation}
 \label{dotTeqneps} \dot{T}_\epsilon:=\left(\sum_{\mathbf{k}\in\mathbb{Z}^2} a^{(1)}_{\mathbf{k},\epsilon}e_{\mathbf{k}},\sum_{\mathbf{l}\in\mathbb{Z}^2} 
 a^{(2)}_{\mathbf{l},\epsilon}e_{\mathbf{l}}\right)
 \end{equation}
 converge in $H^7$ to the unique solution $\Td$ of the optimal linear response problem \eqref{optobs}--\eqref{optobs2}.

\end{theorem}

\begin{proof}
 {The} $(\|\cdot\|_{1,2},\|\cdot\|_{0,3})$  (KL) condition, in particular (KL1), implies that
 there is a monotone increasing
map $\tau : (0,\epsilon_0]\to[0,\infty)$ such that $\lim_{\epsilon \to \infty} \tau(\epsilon)=0$ and the operator 
$\LG_\epsilon:=K_{\epsilon}\circ\LG_0$ satisfies
\begin{align*}
\|\LG_\epsilon-\LG_0\|_{\B^{1,2} \rightarrow \B^{0,3}}\leq \tau(\epsilon),
\end{align*}

{The} $(\|\cdot\|_{2,1},\|\cdot\|_{1,2})$ and $(\|\cdot\|_{1,2},\|\cdot\|_{0,3})$ (KL) conditions and more specifically (KL3) imply the following Lasota--Yorke inequalities adapted respectively to the pair of Banach spaces $(\B^{2,1},\B^{1,2})$ and $(\B^{1,2},\B^{0,3})$:
$\exists \alpha <1$ and $M>\alpha$, {such that for all}
 $m\in \NN$, $\epsilon>0$ and  $f\in \B^{2,1}$ one has
\begin{align}\label{eqlymolli2}
\|\LG_\epsilon^m f\|_{2,1}\leq C\alpha^m \|f\|_{2,1}+CM^m \|f\|_{1,2},
\end{align}
{and for all $f\in\B^{1,2}$, one has}
\begin{align}
\nonumber \|\LG_\epsilon^m f\|_{1,2}\leq C\alpha^m \|f\|_{1,2}+CM^m \|f\|_{0,3}.
\end{align}
Since $\B^{1,2}$ is compactly injected into $\B^{0,3}$, the second uniform Lasota--Yorke inequality above implies that there is $C_0>0$ such that for any $\epsilon>0$ and any $f\in \B^{1,2}$,
\begin{align}
  \nonumber  \|\LG_\epsilon f\|_{1,2}\leq C_0\|f\|_{1,2};
\end{align}
in other words it satisfies item 1) (on $\mathcal{B}^{1,2}$) and items 2) and 3) from Theorem \ref{thm:resolv}. 
Thus, {applying Theorem \ref{thm:resolv}, we see that} \eqref{eq:resolv2} holds, namely for any $\rho>\alpha$ there is $\eta_\rho>0$ such that for any $z\in \mathbb{C}$ with $|z|>\rho$ and $d(z,sp(\LG_0))>0$, there is $C(z)>0$ such that
\begin{align}
\nonumber       \|(z-\LG_\epsilon)^{-1}-(z-\LG_0)^{-1}\|_{V_{1,2}\rightarrow V_{0,3}}\leq C(z)\tau(\epsilon)^{\eta_\rho}.
\end{align}

In order to demonstrate the continuity of $a^{(i)}_{\mathbf{k},\epsilon}$ as $\epsilon\to 0$ we begin to study the $\epsilon\to 0$ limit of the integrand in \eqref{akneqnthmeps}. 
To this end, we let
$g_{k,\epsilon}:= \nabla\cdot\left(  f_\epsilon(\cdot)(D_\cdot T_0)^{-1}(e_{\mathbf{k}},0)(\cdot))\right)\in V^{1,2}$. Recall that according to Theorem 2.3 in \cite{GL06}, there is $0<\sigma<1$ such that for $f\in V_{1,2}$, 
$
\|\LG_0^nf\|_{1,2}\leq \sigma^n\| f\|_{1,2}.
$
Thus $(I-\LG_0)^{-1} : V_{1,2}\mapsto V_{1,2}$ is a continuous linear operator, then one has,
\begin{align}
    &\|(I-\mathcal{L}_\epsilon)^{-1} \mathcal{L}_\epsilon\left(g_{k,\epsilon}\right)-(I-\mathcal{L}_0)^{-1} \mathcal{L}_0\left(g_{k,\epsilon}\right)\|_{0,3}\leq \nonumber \\
    &\|(I-\mathcal{L}_\epsilon)^{-1} -(I-\mathcal{L}_0)^{-1}\|_{V^{1,2}\to V^{0,3}} \| \mathcal{L}_\epsilon \left(g_{k,\epsilon}\right)\|_{1,2}+\|(I-\mathcal{L}_0)^{-1}\|_{V^{1,2}}\|\LG_\epsilon-\LG_0 \|_{\B^{1,2}\to \B^{0,3}}\|g_{k,\epsilon}\|_{1,2}\nonumber\\
    &\leq \tau(\epsilon)^{\eta_\rho}C_0\|g_{k,\epsilon}\|_{1,2} +\tau(\epsilon)\|(I-\mathcal{L}_0)^{-1}\|_{V^{1,2}\to V^{0,3}}\|g_{k,\epsilon}\|_{1,2}.\label{eq:contresp}
\end{align}
What remains is to uniformly bound $\|g_{k,\epsilon}\|_{1,2}$ for $\epsilon\in (0,\epsilon_0]$ to prove that equation \eqref{eq:contresp} goes to $0$ and then to verify the convergence of $(I-\mathcal{L}_0)^{-1} \mathcal{L}_0\left(g_{k,\epsilon}\right)$ in $\B^{0,3}$.

\paragraph{Uniform bound for $\|g_{k,\epsilon}\|_{1,2}$}
To bound $\|g_{k,\epsilon}\|_{1,2}$, we need to control $\|f_\epsilon\|_{2,1}$.  Note that since $f_\epsilon$ is an eigendistribution for the leading eigenvalue $\lambda_\epsilon$ of $\LG_\epsilon$, according to Corollary 1 from \cite{KL99}, for sufficiently small $\epsilon_0$, one has that $\lambda_\epsilon>\rho >\alpha$ for $\epsilon\in(0,\epsilon_0]$ and some $\rho>\alpha$.
The Lasota-Yorke inequality \eqref{eqlymolli2} then ensures, for any $m\in \NN$,
\begin{align*}
    \|f_\epsilon\|_{2,1}& =  \frac{1}{\lambda_\epsilon^m}\|\LG_\epsilon^mf_\epsilon\|_{2,1}\\
    &\leq C\frac{\alpha^m}{\lambda_\epsilon^m} \|f_\epsilon\|_{2,1}+C\frac{M^m}{\lambda_\epsilon^m} \|f_\epsilon\|_{1,2},
\end{align*}
thus for some $m_0$ taken large enough so that $\lambda_\epsilon^{m_0}-C\alpha^{m_0}>0$,
\begin{align}\label{eq:liftreg}
\|f_\epsilon\|_{2,1}\leq \frac{CM^{m_0}}{\lambda_\epsilon^{m_0}-C\alpha^{m_0}}\|f_\epsilon\|_{1,2}.    
\end{align}

As a consequence of Corollary 1 from \cite{KL99},  $f_\epsilon\to f_0$ in $\B^{1,2}$ : indeed, recall from the notation of \cite{KL99} that $\Pi_\epsilon$ is the projection onto $f_\epsilon$, which is the eigenfunction satisfying $\int f_\epsilon=1$. 
Thus from item 1) of Corollary 1 from \cite{KL99}, 
\begin{align}\label{eq:perturbcorun}
\|\Pi_\epsilon(f_0)-\Pi_0(f_0)\|_{1,2}\leq K\tau^{\eta_\rho}(\epsilon)
\end{align}
Let's write $\Pi_\epsilon(f_0)$ as $\Pi_\epsilon(f_0)=\Gamma_\epsilon(f_0)f_\epsilon$ with $\Gamma_\epsilon(\cdot)$ being a linear form. 
Integrating the previous equality, we obtain $\int \Pi_\epsilon(f_0)=\Gamma_\epsilon(f_0)\int f_\epsilon = \Gamma_\epsilon(f_0)$. Since $\int \Pi_0(f_0)=1$, using  \eqref{eq:perturbcorun} we deduce that
$\Gamma_\epsilon(f_0)\to 1$. {We now further conclude from \eqref{eq:perturbcorun} that}
\begin{align}
\label{eq:conv12}
\|f_\epsilon-f_0\|_{1,2}\underset{\epsilon\to 0}{\to} 0.
\end{align}

Returning now to \eqref{eq:contresp}, inequality \eqref{eq:conv12} and \eqref{eq:liftreg} imply that the family $(f_\epsilon)_{\epsilon>0}$ is bounded in $\B^{2,1}$, thus by item iv) of Proposition \ref{lemmulti}, the family $(g_{k,\epsilon})_{\epsilon\in (0,\epsilon_0]}$ is also bounded in $\B^{1,2}$ and 
since $\lim_{\epsilon \to 0} \tau(\epsilon)=0$, the RHS of \eqref{eq:contresp} tends to $0$ 
\paragraph{Verifying that $(I-\mathcal{L}_0)^{-1} \mathcal{L}_0\left(g_{k,\epsilon}\right)$ converges in $\B^{0,3}$.}

We first check the convergence of $g_{k,\epsilon}$ in $\B^{0,3}$.
Recall that $g_{k,\epsilon}=\nabla\cdot\left(  f_\epsilon (\cdot)(D_\cdot T_0)^{-1}(e_{\mathbf{k}},0)(\cdot))\right)$, thus according to items (ii), (iv), and (v) from Proposition \ref{lemmulti}, and the convergence of $(f_\epsilon)_{\epsilon>0}$ toward $f_0$ in $\B^{1,2}$ stated in \eqref{eq:conv12}, $g_{k,\epsilon}\to g_k:= \nabla\cdot\left(  f_0(\cdot)(D_\cdot T_0)^{-1}(e_{\mathbf{k}},0)(\cdot))\right)$ in $\B^{0,3}$.
{
Since $(g_{k,\epsilon})_{\epsilon>0}$ is bounded in $\B^{1,2}$ and $(I-\LG_0)^{-1} : \B^{1,2}\mapsto \B^{1,2}$ is continuous, $(I-\LG_0)^{-1}\LG_0g_{k,\epsilon}$ is bounded in $\B^{1,2}$ and thus one can extract some subsequence converging in $\B^{0,3}$. Let $H$ be the limit in $B^{0,3}$ of such subsequence denoted $(\epsilon_n)_{n\in \mathbb N}$. Then by continuity of $(I-\LG_0) : \B^{p,q}\mapsto \B^{p,q}$ for $p+q=3$,
\begin{align*}
    (I-\LG_0)H&=\lim_{n\to \infty} (I-\LG_0)(I-\LG_0)^{-1}\LG_0g_{k,\epsilon_n}\\
    &=\LG_0 g_k,
\end{align*}
where the above limit is taken in $\B^{0,3}$. We deduce that $H=(I-\LG_0)^{-1}\LG_0 g_k$. By unicity of the limit, we deduce the convergence of $(I-\mathcal{L}_0)^{-1} \mathcal{L}_0\left(g_{k,\epsilon}\right)$ to $(I-\mathcal{L}_0)^{-1} \mathcal{L}_0\left(g_k\right)$ in $\B^{0,3}$.
}
Thus the continuity (see item (v) from Proposition \ref{lemmulti}) of $h\in \B^{0,3} \mapsto \int c\cdot h$ 
allows us to deduce from the expression of 
$a^{(i)}_{\mathbf{k},\epsilon}$ in \eqref{akneqnthmeps} and \eqref{alneqnthmeps} that for any $i\in\{1,2\}$ and any $\mathbf{k}\in \ZZ^2$,
\begin{align*}
\lim_{\epsilon\to 0}a^{(i)}_{\mathbf{k},\epsilon}&= \lim_{\epsilon\to 0}- \frac 1 {\nu_\epsilon  \left(\sum_{m=0}^5  (2\pi)^{2m} |\mathbf{k}|^{2m}_2\right)}\int c\cdot \overline{(I-\mathcal{L}_\epsilon)^{-1} \mathcal{L}_\epsilon\left( g_{k,\epsilon}\right)}  \\
&=a^{(i)}_{\mathbf{k}}.
\end{align*}
Thus by convergence of Fourier series,
$$
\lim_{\epsilon\to 0}\|\Td_\epsilon-\Td\|_{\mathcal{H}}=0.
$$
\end{proof}

\begin{remark}
    \label{Leneremark}
If in \eqref{akneqnthmeps} and \eqref{alneqnthmeps} one replaces the terms that play the role of  $R(\Td)$ in \eqref{eq:response} with the ``metaformula'' (1.4) from \cite{Ba14}, the resulting coefficients would again yield a vector field that converges to the optimal perturbation $\Td$.
\end{remark}

\subsubsection{Examples of kernels that satisfy the hypotheses of Theorem \ref{thmmolli}}
\begin{enumerate}
    \item
When the family $(k_\epsilon)_{\epsilon}$  has compact support ($\mathrm{supp\ }k_\epsilon\subset B(0,\epsilon)$) they satisfy $(\|\cdot\|_{2,1},\|\cdot\|_{1,2})$ and $(\|\cdot\|_{1,2},\|\cdot\|_{0,3})$ (KL) conditions (see Lemma 5.2 from \cite{CF20}).
Moving toward a numerical implementation, one could replace 
$K_\epsilon$ with $K_\epsilon\circ\mathcal{D}_{n(\epsilon)}$ (still for $(k_\epsilon)_{\epsilon}$ a family of kernel with compact support), where $\mathcal{D}_{n(\epsilon)}$ is convolution with the Dirichlet kernel (truncating to the first $n(\epsilon)$ Fourier modes for large enough $n$).
This strategy is justified by Proposition 5.3 \cite{CF20}. 
\item If $T$ satisfies condition (15) from \cite{CF20} (roughly speaking, the norm of $\mathcal{L}_0$ remains bounded under translation of $T$ on the torus), setting $k_\epsilon$ to be a family of $n^{\rm th}$-order Fejer kernels satisfies the hypotheses of Theorem \ref{thmmolli} (see Proposition 6.1  \cite{CF20}), and is particularly useful for numerical implementation.
The use of these kernels is explored in the next subsection.
\end{enumerate}
Having established the above general convergence theorem, we turn our attention to numerically approximating $\mathcal{L}_0$ and $f_0$.
The following subsections detail the computations of the central objects required to estimate the coefficients $a^{(1)}_\mathbf{k}$ and $a^{(2)}_\mathbf{l}$.

\subsection{Numerically estimating the transfer operator $\mathcal{L}_0$} 
\label{sec:Lnumerical}

Because we have a differentiable map on a periodic domain, it is natural to use a Fourier basis to estimate $\mathcal{L}_0$.
The action of $\mathcal{L}_0$ in Fourier space will be represented by {a matrix} acting on the standard Fourier basis $e_\mathbf{k}$ on $\mathbb{T}^2$, for $\mathbf{k}\in F_n:=\{-n/2+1,\ldots,n/2\}^2$.
We desire formal convergence properties for the numerically estimated SRB measure in $\mathcal{B}^{p,q}$ and therefore cannot simply project $\mathcal{L}_0$ onto a truncated basis $\{e_\mathbf{k}\}_{\mathbf{k}\in F_n}$.
Instead we require some mollification prior to projection;  because of its simplicity and strong theoretical properties we use the Fejer kernel scheme of \cite{CF20}.

\subsubsection{Fej\'er kernel mollification}\label{subsecmollifier}
Define the $n$th one-dimensional Fej{\'e}r kernel $K_{n,1}$ on $\mathbb{T}^1 = \mathbb{R} / \mathbb{Z}$ by
\begin{equation*}
  K_{n,1}(x) = \sum_{k=-n}^n \left(1 - \frac{|k|}{n+1} \right)e^{2 \pi i k x}.
\end{equation*}
The $n$th $d$-dimensional Fej{\'e}r kernel $K_{n,d}$ on $\mathbb{T}^d$ is defined by taking a $d$-fold Cartesian product of the one-dimensional kernels:
\begin{equation*}
  K_{n,d}(x) = \prod_{i=1}^d K_{n,1}(x_i),
\end{equation*}
where $x_i$ is the $i$th component of ${x} \in \mathbb{T}^d$.
The Fej\'er kernel is particularly convenient because convolution may be constructed by a simple reweighting of Fourier series:
\begin{equation}\label{eq:fejer_fourier_series}
  (K_{n,d} * f)({x}) = \sum_{\substack{\mathbf{k} \in \mathbb{Z}^d \\ |{\mathbf{k}}|_\infty \le n}} \prod_{i=1}^d \left(1 - \frac{|{k_i}|}{n+1}\right)\hat{f}(k) e^{2\pi i \mathbf{k} \cdot {x}},
\end{equation}

The approximate transfer operators acting on $\mathcal{B}_{p,q}$ are defined by\footnote{We again overload the subscripts on $\mathcal{L}$ and $f$, but the meaning should be clear from the context.}
{
\begin{equation*}
\mathcal{L}_{n}f:=K_{n,2} * (\mathcal{L}_0 f):=(\Pi_n\circ\mathcal{L}_0)f.
\end{equation*}
}
Section 6 \cite{CF20} discusses the satisfaction of the hypotheses (S1) and (S3) (in \cite{CF20}), where this numerical approach is developed.
The operator $\mathcal{L}$ should also satisfy equation (15) in \cite{CF20}.
Corollary 6.4 \cite{CF20}  states that if we convolve the transfer operator with a Fejer kernel (parameterised by $n$ as above), this sequence of operators in $n$ satisfy the $(\|\cdot\|_{p,q},\|\cdot\|_{p-1,q+1})$ (KL) conditions of \cite{CF20}.

\subsubsection{Matrix representation}

\label{sec:Lmatrix}
We consider the operator $\mathcal{L}_{n}$ to act on $\mathcal{B}^{2,1}$.
{Note that the image of $\mathcal{B}^{2,1}$ under $\mathcal{L}_n$ is contained in the span of $\{e_\mathbf{k}\}_{\mathbf{k}\in F_n}$, which is itself invariant under $\mathcal{L}_n$. Therefore when computing the eigenfunction of $\mathcal{L}_n$ corresponding to the largest-magnitude eigenvalue, we may restrict the action of $\mathcal{L}_n$ to $\{e_\mathbf{k}\}_{\mathbf{k}\in F_n}$;  this restricted action has a finite-dimensional representation.}
To obtain a representation of $\mathcal{L}_{n}$ in this basis, we {first fix some notation. Denote the usual $L^2$ inner product $\langle f,g\rangle =\int_{\mathbb{T}^2} f\cdot\overline{g}$. 
One has $f=\sum_{\mathbf{k}\in\mathbb{Z}^2} \langle f,e_\mathbf{k}\rangle e_\mathbf{k}=\sum_{\mathbf{k}\in\mathbb{Z}^2} \hat{f}(\mathbf{k}) e_\mathbf{k}$. 
We now compute}
\begin{eqnarray}
\nonumber [\hat{L}_n]_{\mathbf{k}\mathbf{j}}:=\langle \mathcal{L}_{n} e_\mathbf{k},e_\mathbf{j}\rangle&=&\int_{\mathbb{T}^2}\left(\int_{\mathbb{T}^2} K_{n,2}(x-y)\mathcal{L}_0(e_\mathbf{k}(y))\ dy\right)\overline{e_\mathbf{j}(x)}\ dx\\
\label{used}&=&\widehat{K_{n,2}}(\mathbf{j})\cdot\widehat{\left({e_{-\mathbf{j}}}\circ T\right)}({\mathbf{-k}}),
\end{eqnarray}
where the steps from the first to the second line above are elaborated upon in equation (24) \cite{CF20}.
The Fourier coefficients of $K_{n,2}$ can be immediately read off from \eqref{eq:fejer_fourier_series}.

\subsubsection{Discrete Fourier transform}
The Fourier integrals in \eqref{used} are computed 
 using Julia's two-dimensional discrete fast Fourier transform \verb"fft" in the \verb"FFTW.jl" package \cite{FFTW}.
 Corresponding to the frequency grid $F_n\subset \mathbb{Z}^2$ is a regular spatial grid on $\mathbb{T}^2$ of cardinality $n^2$, which we call a ``coarse grid''.
For $N\ge 4n$, spatial FFTs are taken on a refined regular $N\times N$ grid that contains the coarse grid.
This larger grid is called the ``fine grid'' and enables  accurate estimates of the higher frequencies in the coarse grid.
The cardinality $n^2$ of the coarse grid determines the size of the $n^2\times n^2$ matrix $L_n$, while the cardinality $N^2$ of the fine grid determines the computation effort put into estimating the inner products via the DFT.
In our experiments we will use $n=32$ or $64$, and $N=4n$.
Further details are provided in \cite{CF20}.

\subsection{Resolvent $(I-\mathcal{L}_0)^{-1}$} 
For our approximation of $(I-\mathcal{L}_0)^{-1}$ we wish to compute the action of the resolvent $(I-\mathcal{L}_n)^{-1}$ on the image of $\mathcal{L}_n$ intersected with the space of mean-zero functions $V_0\subset L^2$.
In terms of Fourier modes, we therefore consider $(I-\mathcal{L}_n)^{-1}$ to act on $\mathrm{span}\{e_{\mathbf{k}}:\mathbf{k}\in F_n\setminus (0,0)\}$.
As in the discussion of Section \ref{sec:Lmatrix}, $\mathcal{L}_n$ leaves $\mathrm{span}\{e_{\mathbf{k}}:\mathbf{k}\in F_n\setminus (0,0)\}$ invariant, and the action of  $\mathcal{L}_n$ on this space can be exactly captured by the matrix representation $\hat{L}_n$ (assuming perfect evaluation of Fourier transforms).
Similarly, the action of $(I-\mathcal{L}_n)^{-1}$ on $\mathrm{span}\{e_{\mathbf{k}}:\mathbf{k}\in F_n\setminus (0,0)\}$ is exactly captured by $(I-\hat{L}_n)^{-1}$.
To restrict the matrix representation $(I-\hat{L}_n)$ to the appropriate subspace we simply delete the row and column of $I-\hat{L}_n$ corresponding to the $(0,0)$ mode.
We may then explicitly invert the $(n^2-1)\times (n^2-1)$ matrix $I-\hat{L}_n$.

    \subsection{Observable $c$} We take the FFT of the observation function $c$ evaluated on a fine uniform $N\times N$ spatial grid in $\mathbb{T}^2$  where $N\gtrsim 4n$ and restrict the resulting frequency representation to the lower modes in $F_n$. 
    We further remove the entry corresponding to the mode $(0,0)$ because the integrals in \eqref{akneqnthm} and \eqref{alneqnthm} may be restricted to $V_0$.  We denote the resulting vector by $\hat{c}_n$.
    
    \subsection{Estimate of the SRB measure $f_0$} 
    \label{sec:srb}
    We calculate the leading eigenvector $\hat{v}_n$ of $\hat{L}_n$ using the  \verb"eigs" function in  \verb"Arpack.jl" and assemble the estimate of $f_0$ as $\sum_{\mathbf{k}\in F_n} \hat{v}_{n,\mathbf{k}}e_{\mathbf{k}}$, where $\hat{v}_{n,\mathbf{k}}$ indicates the element of the vector $\hat{v}_n$ corresponding to mode ${\mathbf{k}}$.
   {The complex eigenvector $\hat{v}_n$ of the complex matrix $\hat{L}_n$ typically has an arbitrary complex phase assigned by the numerical eigensolver, so we automatically adjust the phase to maximise the real part of $f_0$.}
    Theorem 6.3 \cite{CF20} guarantees convergence of the SRB measure (in the $\|\cdot\|_{p-1,q+1}$ norm) as $n\to\infty$.

  \subsection{Estimate of the divergence term}
  \label{ss:3terms}
  For the numerical computations, it will be convenient to use an expanded form the divergence term in \eqref{akneqnthm},
\begin{eqnarray}
\nonumber \nabla\cdot\left(  f_0(\cdot)(D_\cdot T_0)^{-1}(e_{\mathbf{k}},0)(\cdot))\right)&=&\langle D_xf_0, (D_xT_0)^{-1}\circ (e_{\mathbf{k}}(x),0)\rangle\\
\nonumber&+& f_0(x)\langle \nabla\cdot (D_xT_0)^{-1},(e_{\mathbf{k}}(x),0)\rangle\\
\label{eq:expanded}&+& f_0(x)\Tr((D_xT_0)^{-1}\circ D_x(e_{\mathbf{k}}(x),0)),
 \end{eqnarray}
 and similarly for \eqref{alneqnthm}.
 This expansion follows in the same way that \eqref{eq:expandeddivrep} is developed from \eqref{eq:divrep}.
  For simplicity, the differentiation for $D_xf_0$ is done by automatic differentiation (AD) applying the \verb"gradient" function in \verb"ForwardDiff.jl" \cite{RLP16} to the estimate of $f_0$  described immediately above, followed by evaluation on the $N\times N$ fine spatial grid. 
  Similarly, lifting $T_0$ to $\mathbb{R}^2$, the derivatives $(D_xT_0)^{-1}$ and $\nabla\cdot(D_xT_0)^{-1}$ are computed by the AD \verb"jacobian" function, and then evaluated on the fine spatial grid. 
 The derivative of the elementary modes of the vector fields $D_x(e_{\mathbf{k}},0)$ and $D_x(0, e_{\mathbf{l}})$ are trivially defined directly and evaluated on the fine spatial grid.
 
 \subsection{Putting everything together} 
 \label{sec:together}
 For each fixed $\mathbf{k}\in F_n$, the three terms calculated in \S\ref{ss:3terms} are added together and a two-dimensional FFT is taken on the fine $N\times N$ spatial grid. 
 The resulting two-dimensional frequency array is restricted to the modes in $F_n\setminus (0,0)$ and multiplied by $\hat{L}_n$ {and} $(I-\hat{L}_n)^{-1}$ 
 and $\hat{c}_n$, and summed to obtain a single complex scalar coefficient $a_{\mathbf{k}}^{(1)}$ or $a_{\mathbf{l}}^{(2)}$.

To obtain a straightforward speedup when looping over modes $\mathbf{k}$ and $\mathbf{l}$ in $F_n$, all components of \eqref{akneqnthm}, \eqref{alneqnthm}, and \eqref{eq:expanded} that do not depend on $\mathbf{k}$ or $\mathbf{l}$ can be pre-computed outside the loop.
This includes the construction and multiplication of $\hat{c}_n$, $(I-\hat{L}_n)^{-1}$,
and $\hat{L}_n$, and all relevant constructions and multiplications of the terms in \eqref{eq:expanded} described in \S\ref{ss:3terms}.

\subsection{Convergence of the Fej\'er kernel algorithm}

For a given Fourier order $n$, the numerical scheme presented in sections \ref{sec:Lnumerical}--\ref{sec:together} produces an approximate optimal vector field 
$$
\Td_n:=\left(\sum_{\mathbf{k}\in\mathbb{Z}^2} a^{(1)}_{\mathbf{k},n}e_{\mathbf{k}},\sum_{\mathbf{l}\in\mathbb{Z}^2} a^{(2)}_{\mathbf{l},n}e_{\mathbf{l}}\right),
$$
 where the coefficients $a^{(1)}_{\mathbf{k},n}$ and $a^{(2)}_{\mathbf{l},n}$ are given by the formulae \eqref{akneqnthmeps} and \eqref{alneqnthmeps}, replacing $\mathcal{L}_\epsilon$ with $\mathcal{L}_{n}$ from Section \ref{sec:Lmatrix} and $f_\epsilon$ with $f_n$ from section \ref{sec:srb}.

We now state a corollary of Theorem \ref{thmmolli} that guarantees that the $\Td_n$ converge to the true optimal $\Td$ under increasing numbers of Fourier modes. 
Because the Fej\'er kernel is not locally supported, we require an additional hypothesis on $\LG_0$.
For $f\in \mathcal{B}^{p,q}$ with $p,q\in \NN$ such that $p+q=3$, let $\tau_y(f)(x)=f(x+y)$ denote the translation by $y$.
We assume that the for any $y\in \TT$, the operator $\tau_y\LG_0 : \B^{p,q}\to \B^{p,q}$ is well defined
and continuous and
\begin{align}\label{eqtranslg}
    \sup_{y\in \TT}\max\{\|\tau_y\LG_0\|_{2,1},\|\tau_y\LG_0\|_{1,2}\}<\infty.
\end{align}
\begin{cor}\label{corfejermol}
  For $c\in C^3(\TT,\mathbb{R})$, if $\LG_0$ satisfies \eqref{eqtranslg} then $\|\Td_n-\Td\|_{H^7}\to 0$ as $n\to\infty$. 
\end{cor}

\begin{proof}
For large enough $n$,  Corollary 6.4 \cite{CF20} ensures that the operator $\LG_n$ satisfies the (KL) condition in \cite{CF20} on $(\|\cdot\|_{2,1},\|\cdot\|_{1,2})$ and $(\|\cdot\|_{1,2},\|\cdot\|_{0,3})$.
From this fact, we may follow the arguments of the proof of Theorem \ref{thmmolli} to obtain the result.
\end{proof}

As discussed in Remark \ref{Leneremark}, one may remove hypothesis \eqref{eqtranslg} by using a Dirichlet kernel in place of a Fej\'er kernel to truncate to the first $n$ Fourier modes, followed by convolution with a locally supported kernel $k_\epsilon$.
The numerical strategy described in Sections \ref{sec:Lnumerical}--\ref{sec:together} could also be applied to the ``metaformula'' in \cite{Ba14} in the obvious way, retaining the convergence guarantee as discussed in Remark \ref{Leneremark}.

\subsection{Computation time}
To give an example of the scale of the computation, with $n=32$ and $N=128$ (the resolution used in Section \ref{sec:casestudy}), we have $32^2=1024$ modes and therefore $2048$ coefficients are required in total for \eqref{akneqnthm} and \eqref{alneqnthm}. 
Each coefficient requires evaluating the three terms in \eqref{eq:expanded} on a fine grid of $(4\times 32)^2=16384$ points, followed by a two-dimensional FFT on this fine grid.
Running Julia 1.11 on a laptop with an Intel Core Ultra 7 processor, the two main loops over $\mathbf{k}$ and $\mathbf{l}$ each take about 6 seconds with standard Julia multithreading (about 5ms per coefficient).
Computations at higher resolutions are also very tractable, e.g.\ with $n=64$ and $N=256$, each of the two main coefficient loops take about 110 seconds (about 26ms per coefficient).

The two-dimensional example in \cite{GaNi25} used a lower resolution of $n=15$ and is reported to take 863 seconds (about 1.92 seconds per coefficient).  
Thus, at least in the two-dimensional case, our approach demonstrates a speedup by a factor of around three orders of magnitude.

\section{Case studies}
\label{sec:casestudy}
In this section we illustrate the efficacy and utility of our approach through two case studies.
Julia code for these computations is available at \url{https://github.com/gfroyland/Optimal-Hyperbolic-Response}.

\subsection{Stabilisation of the hyperbolic fixed point of Arnold's cat map}

In our first case study we let $T_0(x)= (2x_1 + x_2 , x_1 + x_2 )\pmod 1$ be the area-preserving  Arnold cat map with $f_0\equiv 1$.
We select the objective function $c(x_1,x_2)=\cos(2\pi x_1)+\cos(2\pi x_2)$.
The maximum of $c$ occurs at $(x_1,x_2)=(0,0)$, which coincides with the fixed point of $T_0$, and there is a minimum at $(0.5,0.5)$.
Recall we wish to solve \eqref{optobs}--\eqref{optobs2}, maximising $J(\dot{T})=\int c R(\dot{T}) = \int c\cdot \dot{f}$, where $\dot{f}=R(\dot{T})$ is the derivative of $\frac{d}{d\delta}f_\delta|_{\delta=0}$.
To increase the expectation of $c$, mass should be moved away from $(x_1,x_2)=(0.5,0.5)$ and toward the origin.
One might guess that a suitable perturbation $\dot{T}$ tries to make the origin more attracting that it currently is. 
We shall see that this is indeed the case, but that the optimal perturbation $\dot{T}$ does more, making profitable use of most of the phase space.
To estimate the optimal $\dot{T}$ we use the ansatz \eqref{dotTeqn}, and compute the Fourier coefficients $a_{\mathbf{k}}^{(1)}$ and $a_{\mathbf{k}}^{(2)}$ in \eqref{akneqnthm}--\eqref{alneqnthm}.
Note that because $T_0$ is linear, the first two terms on the RHS of \eqref{eq:expanded} are zero.
Because the $H^7$ norm of $e_{\mathbf{k}}$ grows like $|\mathbf{k}|^7$, if we were to use the standard $H^7$ norm the optimal $\dot{T}$ would consist almost exclusively of the very lowest order modes, either $e_{(0,0)}, e_{(1,0)},$ or $e_{(0,1)}$.
This effect can be partially seen in Figures 4, 6, and 9 
 in \cite{GaNi25}, where the standard Sobolev norm is used, and very regular optimal vector fields are obtained.
We therefore introduce a scaled Sobolev norm $\|\cdot\|_{H^7,\gamma}$ where $\|e_{\mathbf{k}}\|_{H^7,\gamma}^2 = \sum_{m=0}^7(2\pi\gamma)^{2m}|\mathbf{k}|_2^{2m}$, for $0<\gamma\le 1$.
In our experiments we choose $\gamma=0.02$, providing a unit ball that allows higher modes to become more active in the optimal solution.
This yields sufficient variability in the optimal vector field to provide a dynamical interpretation of the solution. 

Figure \ref{linearvf}
shows the optimal vector field $\dot{T}$ computed with $n=32$ Fourier modes in each coordinate direction as described in Section \ref{sec:numerics}.
\begin{figure}
       \centering
\includegraphics[width=0.49\linewidth]{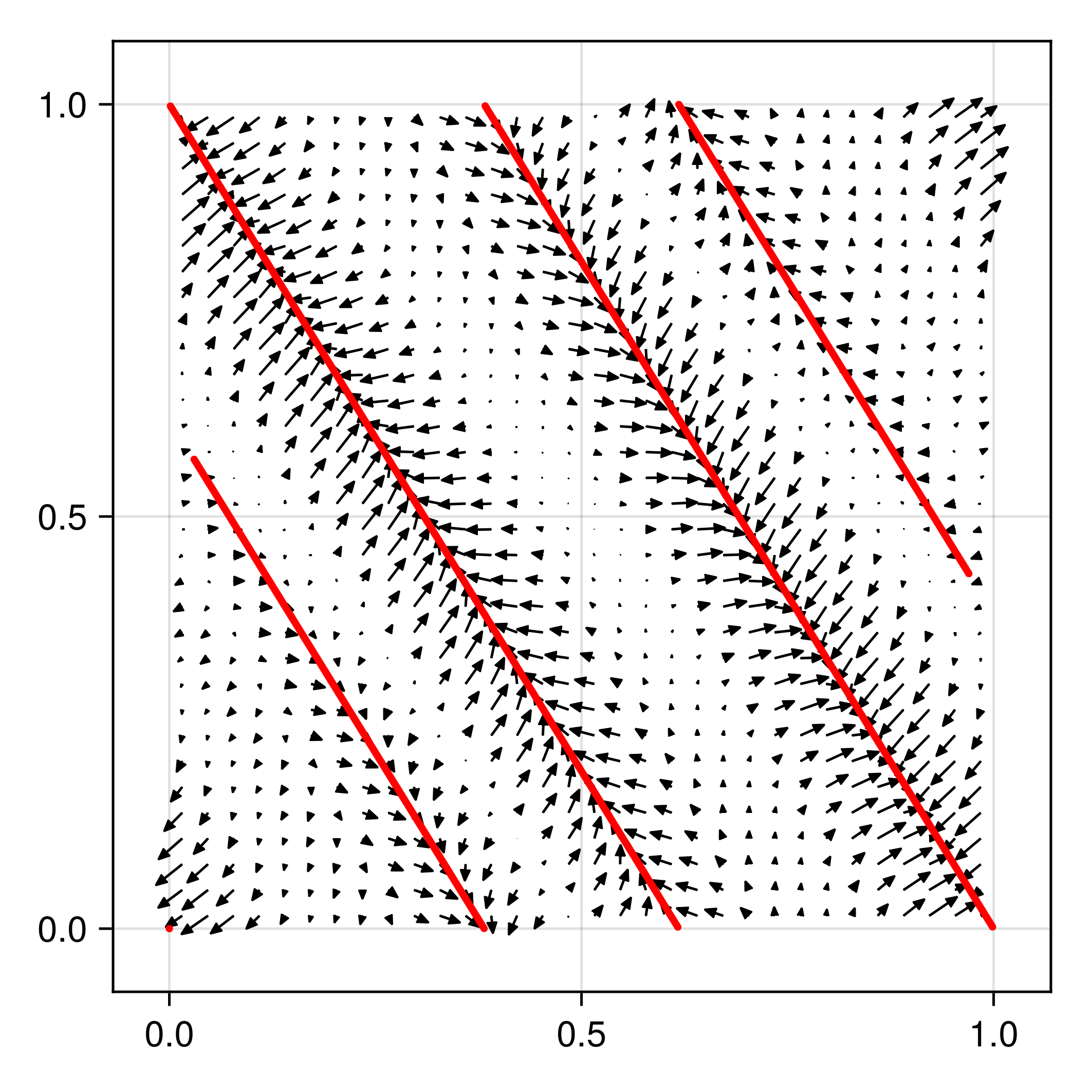}
\caption{Optimal vector field computed with $n = 32$ Fourier modes along each coordinate
direction, using a Sobolev norm scaling factor $\gamma=0.02$. The red lines are a segment of the stable manifold of the fixed point at the origin. 
A uniform scaling of the length of the vectors in this figure has been made for improved visualisation.}
    \label{linearvf}
\end{figure}
Overlaid on the vector field in Figure \ref{linearvf} is a segment of the stable manifold for the fixed point at the origin.
One can immediately see that the vector field pushes a neighbourhood of this stable manifold closer to the stable manifold.
The vector field in this neighbourhood is also partly directed along the stable manifold in the ``direction of travel'' toward the fixed point.
In a neighbourhood of the fixed point itself the vector field pushes almost orthogonally to the stable manifold, and counters the expanding direction of $T_0$ by pointing toward the origin.
This clever strategy means that if a perturbed map $T_\delta=T_0+\delta \dot{T}$ were formed for some small $\delta>0$, the effect of $\dot T$ would be to  entrain mass toward a neighbourhood of the origin under repeated iteration of $T_\delta$.
This is indeed the case;  Figure \ref{linearsrbpert} illustrates how the SRB measure of $T_\delta$ is localised nearby the maximum of the objective function $c$.
\begin{figure}
       \centering
\includegraphics[width=0.49\linewidth]{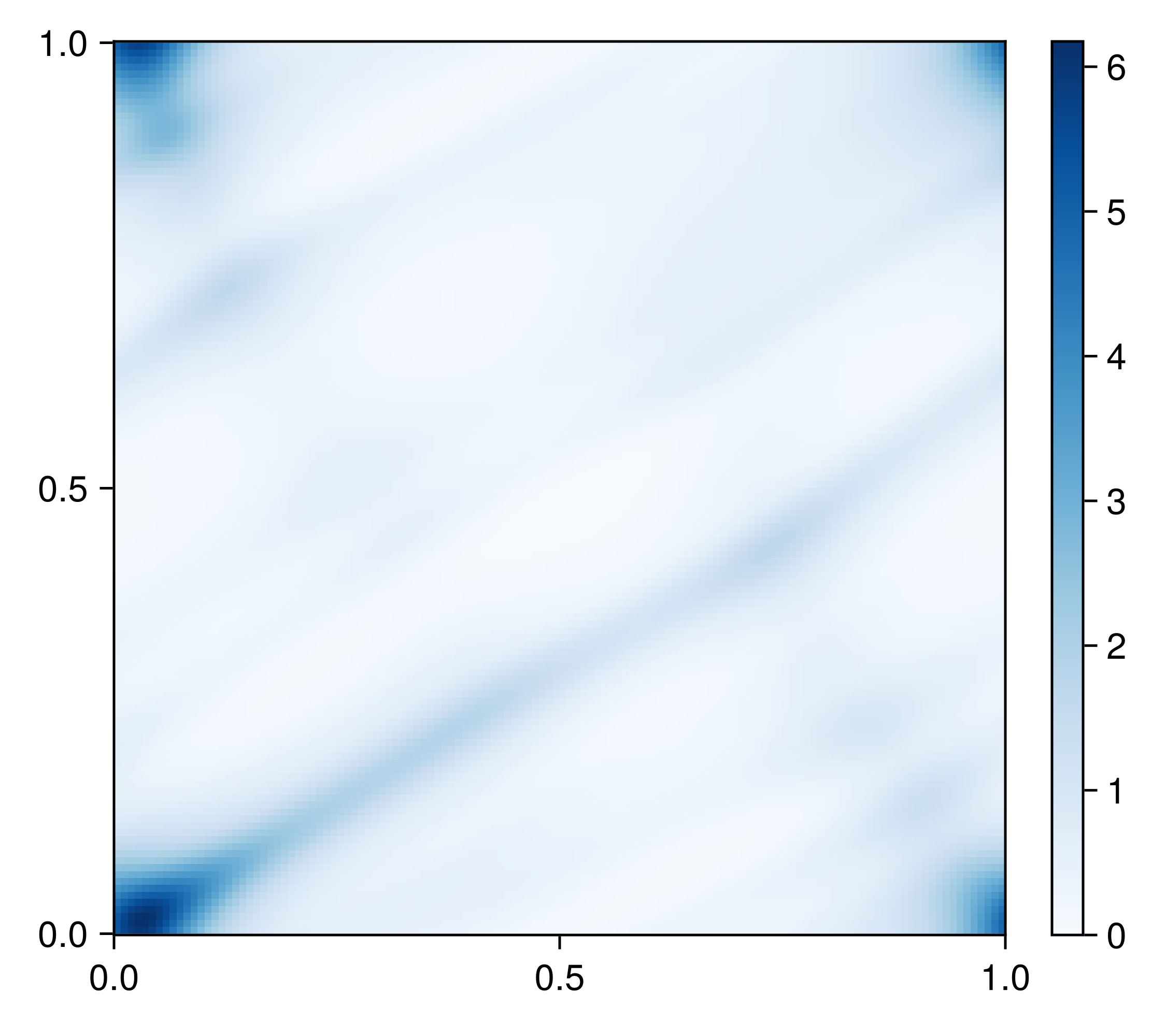}
    \caption{SRB measure estimate of $T_\delta = T_0 +\delta\cdot \dot{T}$ for some positive $\delta$, where the average norm of $\delta\cdot\dot{T}$ is approximately 2\% of the domain diameter. Note the concentration of mass about the fixed point near the origin.}
    \label{linearsrbpert}
\end{figure}
For the choice of $\delta$ in Figure \ref{linearsrbpert}, one has $\delta\int_{\mathbb{T}^2} \|\dot{T}(x)\|\ dx \approx 0.0202.$
In other words, the average norm of $\delta\dot{T}$ is 0.0202, corresponding to an average perturbation of $T_0$ by around 2\% of the domain diameter.

\subsection{Localising nearby a hyperbolic period-two orbit of a nonlinear Anosov map}

In our second case study we modify Arnold's cat map to obtain a nonlinear Anosov map $$T_0(x)= (2x_1 + x_2 + 2\Delta  \cos(2\pi x_1), x_1 + x_2 + \Delta \sin(4\pi x_2 + 1))\pmod 1$$ for $\Delta=0.01$; see \cite{CF20} for a proof that $T_0$ is Anosov.
The map $T_0$ has two hyperbolic period-two orbits, and we focus on the period-two orbit $$(p_1,p_2):\approx\{(0.1796,
 0.4023),(0.7877,
 0.5852)\}.$$
We select the objective function 
\begin{equation*}
c(x_1,x_2)=\sum_{k\in \ZZ^2} \exp\left(-\frac{\|x-(p_1+k)\|^2}{0.1^2}\right) + \exp\left(-\frac{\|x-(p_2+k)\|^2}{0.1^2}\right),
\end{equation*}
namely a sum of two localised Gaussians centred on each of the period-two points, where the norm $\|\cdot\|$ is the periodic Euclidean norm on the torus  and the summation over $k$ ensures $c$ is $C^\infty$ on the torus\footnote{Because $\exp(-1/0.1^2)<10^{-43}$, in the numerical implementation we instead use only the $(0,0)$ term in the sum, namely
$c(x_1,x_2)= \exp(-\|x-p_1\|^2 / 0.1^2) + \exp(-\|x-p_2\|^2/0.1^2)$, incurring a negligible error.}.   
The maxima of $c$ obviously occur at $p_1$ and $p_2$.
By maximising $J(\dot{T})=\int c R(\dot{T}) = \int c\cdot \dot{f}$ to increase the expectation of $c$ under infinitesimal perturbations $\dot{T}$, we therefore expect mass to be concentrated near the period-two orbit.

Figure \ref{nonlinearvf}
shows the optimal vector field $\dot{T}$ computed with $n=32$ Fourier modes in each coordinate direction as described in Section \ref{sec:numerics}.
\begin{figure}
       \centering
\includegraphics[width=0.49\linewidth]{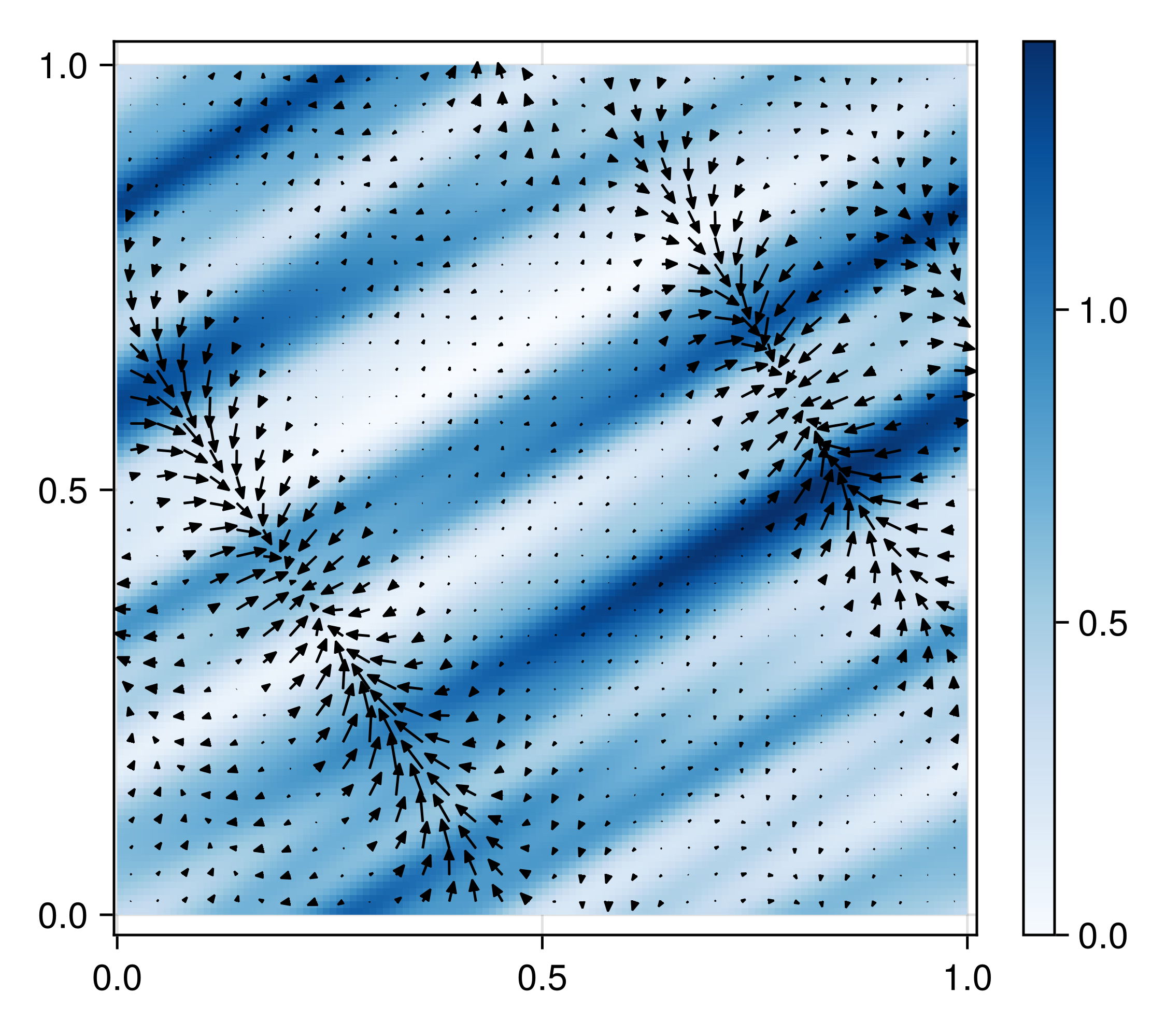}
\caption{Optimal vector field computed with $n=32$ Fourier modes along each coordinate direction, using a Sobolev norm scaling factor $\gamma=0.02$. The blue background is an estimate of the SRB measure of the nonlinear Anosov map $T_0$.
A uniform scaling of the length of the vectors in this figure has been made for improved visualisation.}
    \label{nonlinearvf}
\end{figure}
Because $T_0$ is nonlinear, all terms in \eqref{eq:expanded} contribute to the computation of the coefficients $a_\mathbf{k}^{(1)}$ and $a_\mathbf{l}^{(2)}$.
We use the Sobolev norm $\|\cdot\|_\gamma$ with $\gamma=0.02$ as in the previous example.
Overlaid on the vector field in Figure \ref{nonlinearvf} is an estimate of the SRB measure of $T_0$, computed as described in Section \ref{sec:srb}.
The optimal vector field $\dot{T}$ displays convergent behaviour centred at each of the period-two orbits and entrainment to modest segments of the local stable manifolds of each of the period-two points.
In contrast to Figure \ref{linearvf}, the optimal $\dot{T}$ in Figure \ref{nonlinearvf} is relatively inactive away from local segments of the stable manifolds of the two period-two points.
This is likely due to the different form of the observation $c$.
In the current example, $c$ is indifferent to the presence of mass outside local neighbourhoods of the period-two points, while in the previous example the observation $c$ penalised mass more if it lay farther from the fixed point, thus encouraging a stronger ``sweeping out'' of mass from a larger fraction of the domain. 

Figure \ref{nonlinearsrbpert} shows the SRB measure of $T_\delta=T_0 +\delta \dot{T}$, using the optimised $\dot{T}$.
\begin{figure}
       \centering
\includegraphics[width=0.49\linewidth]{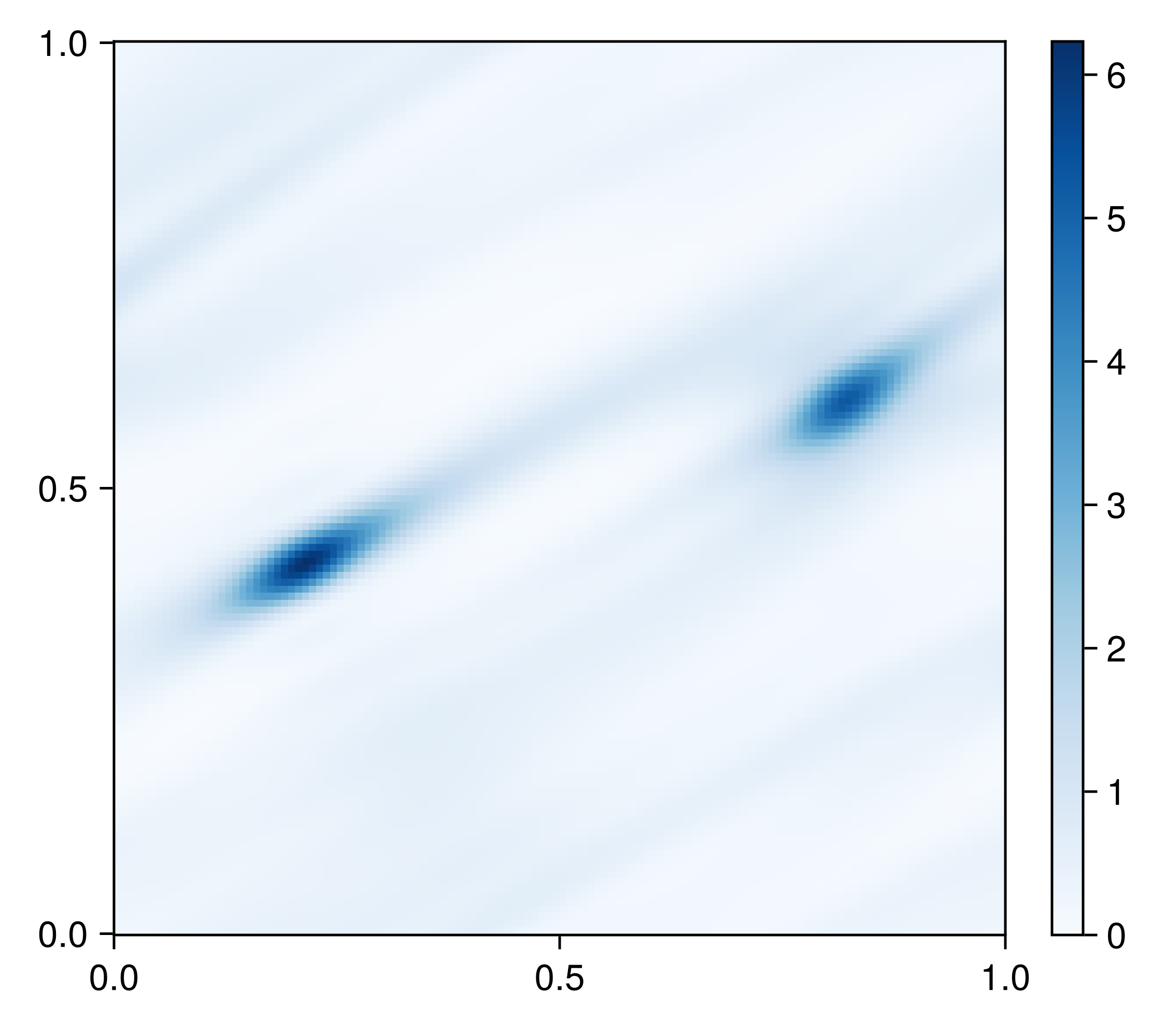}
\caption{Estimate of the SRB measure of $T_\delta = T_0 +\delta\cdot \dot{T}$ for some positive $\delta$, where the average norm of $\delta\cdot\dot{T}$ is approximately 1.2\% of the domain diameter. Note the strong concentration of mass about the period-two orbit.}
    \label{nonlinearsrbpert}
\end{figure}
Note that the SRB measure of $T_\delta$ strongly localises nearby the maxima of the objective function $c$, which occur at the period-two orbit of $T_0$. 
For the choice of $\delta$ in Figure \ref{nonlinearsrbpert}, one has $\delta\int_{\mathbb{T}^2} \|\dot{T}(x)\|\ dx \approx 0.0122.$
In other words, one may accomplish this dramatic targeted change in the SRB measure through a relatively small smooth optimal perturbation to $T_0$ where the average
perturbation $\delta\dot{T}$ is around 1.2\% of the domain diameter.

\newpage
\section{Acknowledgements}
The research of GF and MP is supported by an Australian Research Council Laureate Fellowship (FL230100088). The authors  acknowledge the helpful comments of an anonymous referee.

\bibliographystyle{abbrv}
\bibliography{refs}

\end{document}